\documentclass[english, 11pt]{amsart}

\usepackage{euscript,graphicx,epstopdf,amscd,amsgen,amsfonts,amssymb,latexsym,amsmath,amsthm,graphicx,mathrsfs,times,color,overpic,fourier,array}
\usepackage[all]{xy}
\usepackage{enumitem}
 \usepackage[colorlinks=true]{hyperref}
\usepackage{xcolor}

\usepackage[T1]{fontenc}

\newtheorem{theoremalph}{Theorem}
\newtheorem{corollaryalph}[theoremalph]{Corollary}

\newtheorem*{proposition*}{Proposition}
\newtheorem*{corollary*}{Corollary}
\newtheorem*{claim*}{Claim}
\newtheorem*{remark*}{Remark}
\newtheorem*{problem*}{Problem}
\newtheorem{theorem}{Theorem}[section]

\newtheorem{proposition}[theorem]{Proposition}

\newtheorem{lemma}[theorem]{Lemma}

\newtheorem{claim}[theorem]{Claim}

\theoremstyle{definition}
\newtheorem{notation}[theorem]{Notation}
\newtheorem{definition}[theorem]{Definition}
\newtheorem{remark}[theorem]{Remark}

\numberwithin{equation}{section}

   \def\cM{\mathcal{M}} \def\cS{\mathcal{S}}

\newcommand{\interior}{\operatorname{int}}
\newcommand{\supp}{\operatorname{supp}}
\newcommand{\diff}{\operatorname{Diff}}
\newcommand{\orb}{\operatorname{Orb}}

\newcommand{\per}{\operatorname{Per}}

\newcommand{\eqdef}{\stackrel{\scriptscriptstyle\mathrm{def}}{=}}

\DeclareMathOperator{\RT}{RT}
\DeclareMathOperator{\diag}{diag}

\begin{document}
	
\title{Lyapunov spectrum of homoclinic classes}

\author[L.J. D\'iaz]{Lorenzo J. D\'\i az}
\address{Departamento de Matem\'atica PUC-Rio, Marqu\^es de S\~ao Vicente 225, G\'avea, Rio de Janeiro 22451-900, Brazil}
\email{lodiaz@puc-rio.br}

\author[K.~Gelfert]{Katrin~Gelfert}
\address{Instituto de Matem\'atica Universidade Federal do Rio de Janeiro, Av. Athos da Silveira Ramos 149, Cidade Universit\'aria - Ilha do Fund\~ao, Rio de Janeiro 21945-909,  Brazil}\email{gelfert@im.ufrj.br}

\author[X. Wang]{Xiaodong Wang}
\address{School of Mathematical Sciences,  CMA-Shanghai, Shanghai Jiao Tong University, Shanghai, 200240, P.R. China}\email{xdwang1987@sjtu.edu.cn}

\author[J. Yang]{Jiagang Yang}
\address{Departamento de Geometria, Instituto de Matem\'atica e Estat\'istica, Universidade Federal Fluminense,
S\~ao Domingos, Niter\'oi - RJ, 24210-200,  Brazil}\email{jiagangyang@id.uff.br}

\thanks{LD, KG, and JY were supported [in part] by 
CAPES -- Finance Code 001, by 
CNPq-grants  310069/2020-3, % bolsa Lorenzo
305327/2022-4,  and % bolsa Katrin
312054/2023-8, % bolsa JY 
CNPq Projetos Universais 401737/2025-0,  % Universal Lorenzo 
404943/2023-3, % Universal Pacifico
and 401737/2025-0, % Univ Gugu
 FAPERJ - Carlos Chagas Filho Foundation for Research Support of the State of Rio de Janeiro, 
% E-26/211.313/2021, % Gugu
 E-26/211.313/2021, % JY
 E-26/204.046/2024, % Lorenzo CNE
  E-26/200.371/2023,  %Katrin CNE
  and  E-16/2014 INCT/FAPERJ.  % lorenzo inctmat
XW was partially supported by National Key R\&D Program of China (2021YFA1001900). 
JY was also supported by Capes  MATH-AMSUD, PRONEX and NSFC 12271538.
}


\maketitle

\begin{abstract}
We study the Lyapunov spectrum of the ergodic measures of isolated homoclinic classes of $C^1$-generic diffeomorphisms. We show that this spectrum has nonempty interior and that any vector in its interior is the spectrum of some ergodic measure fully supported on the homoclinic class. We also discuss the averaged Lyapunov spectrum of  homoclinic classes (an extension of the Lyapunov graph).
\end{abstract}

\begin{flushright}{\tiny{{To professor Lan Wen, on the occasion of his 80th birthday}}}
\end{flushright}

\section{Introduction}

\subsection{Specification \emph{versus} $C^1$-generic approach}

In his seminal work~\cite{Sig:74}, Sigmund proved that systems with the specification property have dense {\em periodic measures} (that is, ergodic measures supported on periodic orbits) in the space of invariant measures endowed with the weak$^\ast$ topology. More broadly, specification, which allows orbit-gluing constructions and is closely related to hyperbolicity in the differentiable setting, yields a detailed description of the space of invariant measures: periodic measures are dense and ergodic measures form a residual subset (see Sigmund~\cite{Sig:70,Sig:74} and Bowen~\cite{Bow:08}). In the uniformly hyperbolic setting, through symbolic coding, one obtains a description of the rich ergodic structure of basic sets and their measures. 

Beyond the specification framework, this structure may break down. The density of periodic measures can fail, and invariant measures may arise that are not approximated by ergodic ones. Examples of these phenomena appear in~\cite{DiaHorRioSam:09} and its subsequent developments~\cite{DiaGelRam:14,DiaGel:16,Diaetal:19}, involving porcupine-like horseshoes and exposed pieces of dynamics, where distinct hyperbolic structures coexist without ``gluing'' coherently. A central question is thus to understand how much of the description for systems with the specification property survives beyond specification, particularly in the differentiable setting.

Beyond the weak$^\ast$ density of periodic measures, one can investigate finer ergodic properties and further quantifiers. Here, we focus on the Lyapunov spectrum of a measure that considers all its Lyapunov exponents simultaneously. We aim to understand the full range of such possible values that can be obtained, going beyond the hyperbolic setting. Here, we focus on the class of measures supported on a given homoclinic class. Indeed, for \emph{$C^1$-generic diffeomorphisms} (that is, in a $C^1$-residual subset of the space of $C^1$-diffeomorphisms),  homoclinic classes are the natural substitute of basic sets in Smale's spectral decomposition theorem~\cite{Sma:63} (see \cite{BonCro:04}). 

Note that, in general,  homoclinic classes may fail to be hyperbolic and exhibit more intricate features. They may contain periodic points with different {\em $\mathrm{s}$-indices} (that is, different dimensions of the stable bundle). Moreover, a homoclinic class may even strictly contain smaller homoclinic classes. Such structural complexity is reflected in corresponding pathologies in the space of invariant measures. 

A precursor for our results is \cite[Corollary 2]{Abdetal:07}, which states that for homoclinic classes of $C^1$-generic diffeomorphisms, the closure of the Lyapunov spectra of the periodic measures supported on the homoclinic class is a convex set. Here we provide a more detailed description of that closure. 

In this $C^1$-generic context, in \cite{AbdBonCro:11} the above mentioned ``orbit-gluing constructions'', typical of settings with specification, is replaced by a perturbative approach, as explained in \cite[Introduction]{AbdBonCro:11}. This involves tools such as Pugh's Closing Lemma \cite{Pug:67}, Hayashi's Connecting Lemma \cite{Hay:97,WenXia:00}, and their extensions in~\cite{BonCro:04}). Rather than relying on intrinsic dynamical concatenation properties, one exploits the flexibility of $C^1$-perturbations to create abundant periodic orbits and to organize them through homoclinic relations.

This philosophy can already be traced back to the work of Ma\~n\'e, who claimed,
without providing a complete proof, that the density of periodic measures among invariant measures for $C^1$-generic diffeomorphisms follows from his Ergodic Closing Lemma \cite{Man:82}. A complete proof was later given in~\cite[Item (i) of Theorem 3.8]{AbdBonCro:11}. In fact,~\cite{AbdBonCro:11} substantially extends Sigmund's results to the $C^1$-generic setting and establishes the foundations for subsequent developments of the theory. In this way, several features, which are reminiscent of specification, are recovered without orbit concatenation. In our study, as well as in related works, many of the tools developed to study topological and geometrical aspects, such as those in \cite{BonDiaPuj:03, Bonetal:03, Abdetal:07}, together with extensions of shadowing lemmas as in \cite{Gan:02}, play a key role in obtaining ergodic counterparts.
 
 However, even in the $C^1$-generic setting, where periodic measures are abundant, the periodic approximation perspective has intrinsic limitations when trying to go beyond weak$^\ast$ denseness. Indeed, note that weak$^\ast$ limits of periodic measures need not be ergodic. A finer analysis, therefore, requires understanding accumulation points of periodic measures, and in particular sequences whose limits remain ergodic.
To address this issue, our main tool is to use the method and criterion introduced in~\cite{Goretal:05} and later axiomatized in~\cite{BonDiaGor:10}, which ensures that the limit of a suitably controlled sequence of periodic measures is an aperiodic ergodic measure. The resulting measures are called {\em axiomatized GIKN measures}, 
as coined in \cite{KwiLac:}. These measures provide a convenient substitute for periodic ones, as they have low dynamical complexity: they belong to the class of rank-one measures~\cite{KwiLacTri:} and therefore have zero entropy, see Remark \ref{r.GIKNmeasures} for further details.
 
\subsection{Lyapunov spectrum of homoclinic classes}
 
Our main object of study is the {\em Lyapunov spectrum} of invariant and ergodic measures supported on homoclinic classes, which we now define precisely.

Throughout the paper, $M$ denotes a compact boundary-less Riemannian manifold of dimension $m\eqdef\dim(M)\ge 2$. We denote by $\diff^1(M)$ the space of $C^1$-diffeomorphisms of $M$, endowed with the $C^1$-topology.

We denote by $\mathcal{M}(M)$ the space of Borel probability measures on $M$, endowed with the weak$^\ast$ topology. For $f\in\diff^1(M)$, we write $\mathcal{M}_{\rm inv}(f)\subset\mathcal{M}(M)$ for the space of $f$-invariant probability measures and $\mathcal{M}_{\rm erg}(f)\subset\mathcal{M}_{\rm inv}(f)$ for the subset of ergodic ones. When $\Lambda$ is an $f$-invariant compact set, for $\dag \in \{\mathrm{inv}, \mathrm{erg}\}$, we denote by $\mathcal{M}_{\dagger}(\Lambda,f)$ the subset of $\mathcal{M}_{\dagger}(f)$ consisting of measures supported on $\Lambda$.

Given $f\in\diff^1(M)$ and $\mu\in\mathcal{M}_{\rm erg}(f)$, Oseledets' theorem~\cite{Ose:68} yields an $f$-invariant measurable set $\Gamma\subset M$ with $\mu(\Gamma)=1$, together with a measurable $Df$-invariant splitting over $\Gamma$
and numbers,
\begin{equation}
\label{e.oseledetselyapunov}
T_\Gamma M = \bigoplus_{i=1}^{\ell} E^i,
\qquad
\chi_1(\mu)>\cdots>\chi_\ell(\mu),
\end{equation}
such that for every $x\in\Gamma$ and every $0\neq v\in E^i_x$,
\begin{equation}
\label{e.expolyapunov}
\chi_i(\mu)
=
\lim_{n\to\pm\infty}
\frac{1}{n}\log \|Df^n(x)v\|.
\end{equation}
The {\em{multiplicity}} of $\chi_i(\mu)$ is the dimension of the bundle $E^i$.

The {\em{Lyapunov spectrum}} of $\mu\in\mathcal{M}_{\rm erg}(f)$ is the $m$-dimensional vector
\begin{equation}
\label{e.Lyapunovspectrummeasure}
	L(\mu) \eqdef 
	(\lambda_1(\mu),\dots, \lambda_m(\mu))\eqdef
	\big( \underbrace{\chi_1 (\mu),\dots,\chi_1(\mu)}_{\tiny{\mathrm{multiplicity\,} \chi_i(\mu)}},
	\dots, \underbrace{\chi_\ell (\mu),\dots,\chi_\ell(\mu)}_{\tiny{\mathrm{multiplicity\,} \chi_\ell (\mu)}} \big)
		\end{equation}
formed considering
the Lyapunov exponents of $\mu$ listed in nonincreasing order and 
counted  with their corresponding multiplicity (in particular, for each $j=1,\dots, m$ it holds
$\lambda_j(\mu)=\chi_i(\mu)$ for some $i$).
The Lyapunov spectrum of  an invariant measure is obtained by integrating 
the spectra of the ergodic measures of its ergodic decomposition.

Denote by $\per(f)$ the set of periodic points of $f$. For $p\in\per(f)$, we
denote by $\pi(p)$ its minimal  period. The ergodic \emph{periodic measure}
associated to $p$ is
\begin{equation}
\label{e.mup}
\mu_p \eqdef \frac{1}{\pi(p)}\sum_{i=0}^{\pi(p)-1}\delta_{f^i(p)},
\end{equation} 
where $\delta_x$ denotes the Dirac measure
at $x$. 
We denote by
$\mathcal{M}_{\rm per}(f)\subset\mathcal{M}_{\rm erg}(f)$ the set of
\emph{periodic measures}. A measure that is not periodic is called {\em{aperiodic}.}

The \emph{Lyapunov spectrum of the periodic orbit $\orb(p)$} is defined by,  with slight abuse of notation, 
\begin{equation}
\label{e.Lyapunovspectrumppoint}
L(p)\eqdef (\lambda_1(p), \dots,\lambda_m(p))\eqdef L(\mu_p).
\end{equation}

We define the {\em{Lyapunov spectrum}} 
   of an $f$-invariant closed set  $\Lambda$ by  
\begin{equation}
\label{e.Lyapunovspectrumset}
\mathcal L(\Lambda)
\eqdef \overline{\left\{L(\mu)\colon \mu \in \mathcal{M}_{\rm erg} (\Lambda, f)
\right\}}.
\end{equation}
Note that, considering the closure, we are potentially adding vectors that may correspond to the Lyapunov spectrum of nonergodic measures, as well as vectors that may not correspond to any spectrum.

We are interested in the case when $\Lambda=H$ is a \emph{nontrivial homoclinic class}.
Recall that a \emph{homoclinic class} is, by definition, the closure of the transverse intersections between the invariant manifolds of a hyperbolic periodic orbit, or equivalently, the closure of the periodic points homoclinically related to it. In particular, any homoclinic class is an $f$-invariant closed set and periodic points are dense in it  (see Section~\ref{ss.generic} for details). 
A homoclinic class is \emph{nontrivial} if it has at least two different (and hence uncountably many) orbits.
Since, by definition,  the set $H \cap \per(f)$ is dense in $H$ and, for $C^1$-generic diffeomorphisms, every ergodic measure is a weak$^\ast$-limit of periodic measures, it is natural to consider the {\em{periodic  Lyapunov spectrum}} of a homoclinic class defined by
\begin{equation}
\label{e.periodicLyapunovspectrumclass}
\mathcal{L}_{\mathrm{per}}(H) \eqdef  \overline{\left\{ L(\mu) \colon 
\mu \in \mathcal{M}_{\mathrm{per}}(H, f) \right\}}.
\end{equation}

\subsection{Main results}

As observed above, the set $\mathcal{L}(H)$ may
contain vectors that do not correspond to the Lyapunov spectrum of any ergodic
measure. We prove that this is not the case for vectors in the interior of this set.

In what follows, we denote by $\supp(\mu)$ the support of a measure $\mu$ and by
$\interior (S)$ the interior of a set $S$. 

\begin{theoremalph}\label{Thm:Lyapunov-spectrum}
	There exists a residual subset $\mathcal{R}$ of $\,\diff^1(M)$ such that for every $f\in\mathcal{R}$,  
	every nontrivial homoclinic class $H$ of $f$, and every vector
	 $L\in\interior (\mathcal L_{\rm per}(H))$, there exists $\mu\in\mathcal{M}_{\rm erg}(H,f)$ such that 
	\[
	L(\mu)=L  \qquad \text{and} \qquad  \supp(\mu)=H.
	\]
\end{theoremalph}

Recall that a homoclinic class $H$ is said to be {\em{isolated}} if there is a neighborhood $U$ of it (called {\em{isolating neighborhood of $H$}})  such that
\begin{equation}
\label{e.isolatingblock}
H = \bigcap_{i\in \mathbb{Z}} f^i (U).
\end{equation}
Under the hypotheses of isolation, it is possible to give a more accurate description of the two spectra in \eqref{e.Lyapunovspectrumset} and \eqref{e.periodicLyapunovspectrumclass}. 
Recall that  for $C^1$-generic diffeomorphisms, the Lyapunov spectrum of any ergodic measure $\mu$ is the limit of the Lyapunov spectrum of a sequence of periodic measures that weak$^\ast$ converges to $\mu$,
see \cite[Theorems 3.5 and 3.8]{AbdBonCro:11}.
The isolation property implies that the approximating periodic measures are also supported on the class.  Thus
 for every isolated  homoclinic class $H$ of a $C^1$-generic diffeomorphisms it holds
\begin{equation}
\label{e.samedef}
\mathcal{L}(H) = \mathcal{L}_{\mathrm{per}}(H).
\end{equation}
For details, see Remark~\ref{r.equality}.

Assuming that the homoclinic class is isolated, we obtain a strengthened version of the previous result, reminiscent of \cite[Theorem~3.5]{AbdBonCro:11}:

\begin{theoremalph}\label{Thm:isolated-ergodic-denssness}
	There exists a residual subset $\mathcal{R}$ of $\diff^1(M)$ such that every
$f\in\mathcal{R}$,
	every nontrivial and isolated homoclinic class $H$ of $f$,  and  every vector $L\in\interior (\mathcal L(H))$,
	the set 
	\[
	\left\{\mu\in\mathcal{M}_{\rm erg}(H,f)\colon L(\mu)=L \quad \text{and} \quad \supp(\mu)=H\right\}
	\] 
    is dense 
	\[\left\{\mu\in\mathcal{M}_{\rm inv}(H,f)\colon L(\mu)=L\right\}.\]
\end{theoremalph}

\begin{remark}
By \eqref{e.samedef} together with Theorem~\ref{Thm:periodic-interior-spectrum} stated below, it follows that any isolated homoclinic class of a diffeomorphism in Theorems~\ref{Thm:Lyapunov-spectrum} and \ref{Thm:isolated-ergodic-denssness} 
satisfies $\interior (\mathcal{L}_{\rm per}(H))=\interior (\mathcal{L}(H))\neq\emptyset$.
\end{remark}

\begin{remark}[Rank-one measures]\label{r.GIKNmeasures}
Our construction shows that the measures realizing the Lyapunov spectra in Theorems \ref{Thm:Lyapunov-spectrum} and \ref{Thm:isolated-ergodic-denssness} can be chosen axiomatized GIKN ones and hence rank-one measures (see \cite{KwiLacTri:}). Let us observe that,
\[
	\{\text{axiomatized GIKN}\}
	\subset \{\text{rank-one}\}
	\subset \{\text{zero metric entropy}\}.
\]
The first inclusion is proved in \cite{KwiLacTri:}. The fact that axiomatized GIKN measures have zero entropy has previously been proved in \cite{KwiLac:}.
\end{remark}

A particular case of Theorem~\ref{Thm:isolated-ergodic-denssness} occurs when the homoclinic class coincides with the whole manifold. To discuss this situation, recall that a set is \emph{transitive} if it is the closure of the orbit of one of its points. A diffeomorphism $f\in \diff^1(M)$ is \emph{transitive} if $M$ is a transitive set for $f$, and it is \emph{$C^1$-robustly transitive} if it has a neighborhood in $\diff^1(M)$ consisting entirely of transitive diffeomorphisms. We denote by $\RT^1(M)$ the (open) subset of $\diff^1(M)$ formed by $C^1$-robustly transitive diffeomorphisms.

The next result follows from Theorem~\ref{Thm:isolated-ergodic-denssness} together with the following observation. Under the assumption of robust transitivity and $C^1$-generically, the manifold $M$ coincides with a homoclinic class (which is therefore isolated). 
This fact is by now standard and follows from the Connecting Lemma~\cite{Hay:97,WenXia:00}. It is stated explicitly in the three-dimensional case in \cite[Lemma~3.3]{DiaPujUre:99}, and the higher-dimensional case follows by the same argument. Indeed, the residual subset in the corollary can be taken $\mathcal{T}=\mathcal{R}\cap\RT^1(M)$.

\begin{corollaryalph}\label{c.corollary}
	There exists a residual subset $\mathcal{T}$ of $\,\RT^1(M)$ such that 
	for every $f\in\mathcal{T}$ and every $L\in\interior (\mathcal L(M))$, 
	the set
	\[
	\left\{\mu\in\mathcal{M}_{\rm erg}(f)\colon L(\mu)=L \quad \text{and} \quad \supp(\mu)=M\right\}
	\] 
	is dense in
	\[
	\left\{\mu\in\mathcal{M}_{\rm inv}(f)\colon L(\mu)=L\right\}.
	\]
\end{corollaryalph}

A periodic point $p$ whose Lyapunov spectrum $L(p)=\left(\lambda_1(p),  \dots,\lambda_m(p)\right)$ is such that $\lambda_i(p)\neq \lambda_j(p)$ for every $i\neq j$, is said to have {\em{simple spectrum.}} 
The residual set $\mathcal{R}$ in the previous theorems can be chosen such that every nontrivial  homoclinic class $H$ of a diffeomorphism 
 $f\in \mathcal{R}$ contains a dense subset of periodic points with simple spectrum, see \cite[Proposition 2.3]{Abdetal:07}.
The following theorem states further properties for the spectrum of these periodic points:

\begin{theoremalph}\label{Thm:periodic-interior-spectrum}
There exists a residual\,\footnote{The residual set
$\mathcal{R}'$ is contained in the set $\mathcal{R}$ in previous theorems.}
 subset $\mathcal{R}'\subset \diff^1(M)$ such that, for every $f\in\mathcal{R}'$ and every nontrivial homoclinic class $H$ of $f$, the following holds: 
\[
L(q)\in \interior (\mathcal{L}(H))
\qquad \text{for every periodic point } q\in H\text{ with simple spectrum}.
\]
In particular, $\interior\bigl(\mathcal{L}(H)\bigr)\neq\emptyset$.
\end{theoremalph}

\begin{remark}[Further related results]
\label{r.furtherresults}
In \cite{DonHouTia:}, the authors generalize Katok-like horseshoe constructions \cite{Kat:80} by simultaneously approximating arbitrary convex subsets of the space of invariant measures. In a purely hyperbolic context, they apply this framework to a class of transitive Anosov diffeomorphisms admitting a global (not necessarily dominated) splitting that is ``multi average-conformal'',
 meaning that each ergodic measure has constant Lyapunov exponents within each subbundle. They then study the relative interior of the associated Lyapunov spectrum.

In an {\em{a priori}} nonhyperbolic context, \cite[Theorem 3.2]{AviCroWil:21} refines Katok-type constructions by ``eliminating multiplicities'': for a nearby diffeomorphism, it is constructed a linear horseshoe with a dominated splitting into one-dimen\-sio\-nal subbundles. By considering perturbations of this horseshoe and homoclinic relations among its saddles, one obtains open sets of diffeomorphisms whose horseshoes have Lyapunov spectra with nonempty interior.

By contrast, our approach relies on the GIKN method rather than on horseshoe approximations. Theorem~\ref{Thm:periodic-interior-spectrum} below shows that, for every nontrivial homoclinic class $H$ of a generic diffeomorphism $f$, the set $\mathrm{int}(\mathcal{L}(H))$ is nonempty. Outside the generic setting, however, this set may be empty; this happens, for instance, for conformal Anosov toral automorphisms, where the spectrum is a singleton. Variations of this example can be obtained, for instance, by considering skew products whose base dynamics is given by such a toral automorphism.
\end{remark}

Our final result is related to the notion of the \emph{Lyapunov graph},
formulated for periodic measures in \cite{BocBon:12}. Here we extend this notion to arbitrary ergodic measures. The Lyapunov graph encodes the Lyapunov spectrum by considering 
 partial sums of the Lyapunov exponents 
(counted with multiplicity). It was used to study perturbations of Lyapunov exponents, in particular, 
\cite[Theorem~2]{BocBon:12} characterizes the graphs that can be obtained 
via local perturbations of periodic orbits. This is particularly relevant in nonconformal cases, compare Remark~\ref{r.furtherresults}.

To proceed, we need the concept of a dominated splitting, see Section~\ref{ss.hyperdomination} for details. Recall that a dominated splitting\footnote{We order the bundles of a dominated splitting compatibly with the Lyapunov spectrum, that is, $F^1$ denotes the strongest bundle and $F^k$ the weakest one.} over an $f$-invariant set $\Lambda$,
\[
T_\Lambda M = F^1 \oplus \cdots \oplus F^k,
\]
 is called \emph{finest} if none of the subbundles $F^j$ admits a further dominated splitting for $j=1,\ldots,k$ (that is, each $F^j$ is \emph{indecomposable}). 
The finest dominated splitting is well defined and unique, see~\cite[Proposition~4.11]{BonDiaPuj:03}. It is characterised by the property that any dominated splitting over $\Lambda$ is obtained by grouping together (that is, by taking direct sums of) the bundles $F^i$.

Note that given $\mu \in \mathcal{M}_{\mathrm{erg}}( f)$ its Oseledets splitting $
T_\Gamma M = \bigoplus_{i=1}^{\ell} E^i$ in \eqref{e.oseledetselyapunov} 
(over a set $\Gamma$ with $\mu(\Gamma)=1$) may fail to be dominated. The absence of domination has an impact on the Lyapunov spectrum of perturbations of the dynamics (see \cite{BocBon:12} where periodic measures and ``their perturbations'' are considered). To deal with this, we need to look simultaneously at the finest dominated splitting of the support of $\mu$,
\begin{equation}
\label{e.finestdominated}
T_{\supp(\mu)}M=F^1\oplus\cdots\oplus F^k.
\end{equation}
Note that $\supp (\mu)$ is the closure of $\Gamma$. 
Observe also that every bundle of the Oseledets splitting is necessarily contained in some bundle of the splitting in \eqref{e.finestdominated}, hence $\ell\ge k$.
The interesting case occurs when $\ell>k$, as discussed below.

We now introduce the {\em averaged Lyapunov exponents} and the {\em averaged Lyapunov spectrum} of $\mu$. Recall the Lyapunov spectrum $L(\mu)$ of $\mu$ in \eqref{e.Lyapunovspectrummeasure}. 
Set
\[
d_0 \eqdef 0, \qquad 
d_j \eqdef \dim(F^1\oplus\cdots \oplus F^{j}) \quad \text{for } j=1,\dots,k,
\]
and define
\[
\widehat \lambda_j(\mu) \eqdef 
\frac{\sum\limits_{r=d_{j-1}+1}^{d_j}\lambda_r(\mu)}{\dim(F^j)}
=
\frac{\sum\limits_{r=d_{j-1}+1}^{d_j}\lambda_r(\mu)}{d_j-d_{j-1}}
\]
as the average of the Lyapunov exponents whose Oseledets bundles are contained in $F^j$.
Here it is convenient to write the averages using the numbers $\lambda_r(\mu)$, 
which enumerate the Lyapunov exponents with multiplicities. Note that each 
$\chi_i(\mu)$ coincides with some $\lambda_r(\mu)$.
Accordingly, we define the {\em averaged Lyapunov spectrum of $\mu$}
\begin{equation}\label{e.Lhatmu}
\widehat L(\mu) \eqdef 
(\underbrace{\widehat \lambda_1(\mu),\dots,\widehat \lambda_1(\mu)}_{\dim(F^1)},
\dots,
\underbrace{\widehat \lambda_k(\mu),\dots,\widehat \lambda_k(\mu)}_{\dim(F^k)}).
\end{equation}
For periodic measures, this vector recovers the Lyapunov graph in~\cite{BocBon:12}.

Note that $\widehat L (\mu)$ and $L(\mu)$ are different if, and only if, $\ell> k$. Indeed, if $\ell>k$ and, for instance, $F^1$ contains $E^1\oplus E^2$ from the Oseledets splitting \eqref{e.oseledetselyapunov}, then $\widehat \lambda_1(\mu) < \lambda_1(\mu)$. Similarly, for the other coordinates of $\widehat L (\mu)$ and $L(\mu)$.

\begin{theoremalph}\label{Thm:application-Bochi-Bonatti}
	There exists a residual\,\footnote{The residual set
$\mathcal{R}''$ is contained in the set $\mathcal{R}'$ in Theorem~\ref{Thm:periodic-interior-spectrum}} 
subset $\mathcal{R}'' $ of $\diff^1(M)$ such that for every $f\in\mathcal{R}'$, every  nontrivial isolated homoclinic class $H$ of $f$, and every $\mu\in\mathcal{M}_{\rm erg}(H)$ it holds
\[
	\left\{L\colon L=t\cdot L(\mu)+(1-t)\cdot \widehat L(\mu), t\in[0,1]\right\}
	\subset \mathcal{L}(H).
\]
\end{theoremalph}

We will derive Theorem \ref{Thm:application-Bochi-Bonatti} from the above results together with \cite[Theorem 2]{BocBon:12} and~\cite[Theorem 3.8]{AbdBonCro:11}.

\subsection{Further remarks}
We close this introduction by putting our results into the broader context of multifractal analysis. Given a vector-valued function $\Phi\colon M\to\mathbb R^k$ with \emph{continuous} coordinate functions $\phi_1,\ldots$, $\phi_k$, consider the  map
\[
	\Phi_\ast\colon\cM_{\rm inv}(f)\to\mathbb R^k,\quad
	\Phi_\ast(\mu)
	\eqdef \Big(\int\phi_1\,d\mu,\ldots,\int\phi_k\,d\mu\Big).
\]
Note that $\Phi_\ast(\mu)$ is also called the \emph{rotation vector} of $\mu$ (introduced in \cite{Zie:95}, see \cite{Jen:01} for a survey). As $\Phi_\ast$ is affine and continuous and $\cM_{\rm inv}(f)$ is convex and weak$^\ast$ compact, the rotation set $\Phi_\ast(\cM_{\rm inv}(f))$ is compact and convex. Rotation sets provide a natural framework for the study of multifractal analysis; see, for example, \cite{BarSauSch:02}. The analysis of the (multifractal) Lyapunov spectrum is, in general, more delicate, as individual Lyapunov exponents are given by measurable functions, which are typically not continuous. Extending ``additive'' multifractal analysis to the more general ``sub-additive'' context, with applications to the Lyapunov spectrum, was given, for example, in \cite{FenHua:10}.

Providing just a glimpse of natural extensions, one can study either the spectra while simultaneously controlling several dynamical quantifiers. Given one continuous function $\phi\colon M\to\mathbb R$ let 
\[
	\cS(\alpha)
	\eqdef\left\{\mu\in\cM_{\rm erg}(f)\colon \int\phi\,d\mu=\alpha\right\}.
\]
In a uniformly hyperbolic context, for any $\alpha\in\interior(\phi(\cM_{\rm erg}(f)))$ \cite{HouTia:24} proves the existence of measures in
\[
	\left\{\mu\in\cM_{\rm erg}(f)\colon \int\phi\,d\mu=\alpha,h_\mu(f)= \beta\right\},
	\quad \text{ for every }\quad
	\beta\in[0,\sup_{\mu\in\cS(\alpha)}h_\mu(f)],
\]
where $h_\mu(f)$ stands for the entropy of $\mu$.
This type of question was also addressed in \cite{KucWol:16}.

One can also fix one quantifier and simultaneously study the structure of the corresponding ``sections in the space of ergodic measures''. In this direction, in certain partially hyperbolic contexts with one-dimensional center direction $E^{\mathrm{c}}$ and with $\phi$ being the function determining the center Lyapunov exponent $\chi^{\mathrm{c}}(\mu)=\int\phi\,d\mu$ of a measure, in \cite{DiaGelRamZha:} it is shown that for all admissible values $\alpha$,
\[
	\left\{\mu\in\cM_{\rm erg}(f)\colon \chi^{\mathrm{c}}(\mu)=\alpha\right\}
\]
contains a weak$^\ast$ path of measures whose entropy continuously varies between $0$ and any value arbitrarily close to $\sup_{\mu\in\cS(\alpha)}h_\mu(f)$. Moreover, also in a partially hyperbolic context, \cite{CriDia:} shows that each nonempty set $\cS(\alpha)$ densely contains axiomatized GIKN measures.

This paper is organized as follows. In Section~\ref{s.preliminaries}, we recall the main ingredients and concepts used throughout the work (dominated splittings, Lyapunov exponents and Oseledets splittings, axiomatized GIKN measures, and properties of homoclinic classes of $C^1$-generic diffeomorphisms). In the subsequent sections, we prove the main theorems following alphabetical order. In particular, in Section \ref{ss.generic}, we construct the residual subset $\mathcal{R}$ in Theorems \ref{Thm:Lyapunov-spectrum} and \ref{Thm:isolated-ergodic-denssness}. 

\begin{notation}
Throughout this paper, convergence of a sequence of measures is always understood with respect to the weak$^\ast$ topology, accordingly, we will omit explicit reference to it.

Similarly, the space $\diff^1(M)$ is always endowed with the $C^1$-topology, and we will not mention this explicitly. Moreover, all perturbations are in the $C^1$-topology and assumed to be arbitrarily small.

For the sake of precision, these specifications will be stated explicitly in theorems, propositions, lemmas, and definitions, though they will be omitted in the remainder of the paper.
\end{notation}

\section{Preliminaries}\label{s.preliminaries}

We now introduce the main ingredients of this paper:
hyperbolicity and dominated splittings (Section \ref{ss.hyperdomination}), 
Lyapunov exponents and Oseledets splittings (Section~\ref{ss.lyapunov}),
convergence to ergodic measures and the Axiomatized GIKN measures (Section~\ref{ss.GIKN}),
and
$C^1$-generic properties of diffeomorphisms (Section \ref{ss.generic}).
Throughout this section, we fix $f\in\diff^1(M)$. Recall that $m=\dim (M)$.
\subsection{Dominated splittings}
\label{ss.hyperdomination}

\begin{definition}[Hyperbolicity]\label{Def:hyperbolicity}
	An $f$-invariant compact set  $\Lambda$ is \emph{hyperbolic} if there are a continuous $Df$-invariant splitting $T_\Lambda M=E^{\mathrm{u}}\oplus E^{\mathrm{s}}$ and  constants $\lambda>0$ and $C>1$ such that for every $x\in \Lambda$ and every $n\in\mathbb{N}$, it holds:
	\begin{itemize}	
		\item
		$\|Df^{n}(v^{\mathrm{s}})\|\leq C\cdot e^{-n\lambda}\|v^{\mathrm{s}}\|$ for every $v^{\mathrm{s}}\in E^{\mathrm{s}}_x$;
		\item
		$\|Df^{-n}(v^{\mathrm{u}})\|\leq C\cdot e^{-n\lambda}\|v^{\mathrm{u}}\|$ for every 
		$v^{\mathrm{u}}\in E^{\mathrm{u}}_x$.
	\end{itemize}
\end{definition}
A periodic point $p$ of $f$ is hyperbolic if, and only if, its orbit $\orb(p)$ is a hyperbolic set.

\begin{definition}[Dominated splitting]\label{Def:DS}
Let  $\Lambda$ be an $f$-invariant compact set. 
A continuous $Df$-invariant splitting with two nontrivial bundles
$
T_\Lambda M = F \oplus E
$
is called \emph{dominated} if there exists a positive integer $N$ such that for every $x \in \Lambda$,
\[
\|Df^{N}|_{E_{x}}\| \cdot \|Df^{-N}|_{F_{f^{N}(x)}}\| < \frac{1}{2}.
\]
More generally, a $Df$-invariant splitting
$
T_\Lambda M = E^1 \oplus \cdots \oplus E^k,
$
with $k \geq 2$ nontrivial bundles, is called \emph{dominated} if for every $1 \leq i < k$, the two-bundle splitting
\[
(E^1 \oplus \cdots \oplus E^i) \oplus (E^{i+1} \oplus \cdots \oplus E^k)
\]
is dominated.
\end{definition}

Note that we order the bundles of a dominated splitting in decreasing strength: $F^1$ is the strongest bundle and $F^k$ the weakest one.

For further background on dominated splittings, we refer to \cite[Appendix B1]{BonDiaVia:05}.

\subsection{Lyapunov exponents and exterior products} \label{ss.lyapunov}

Let $\mu$ be an ergodic measure. Recall its Oseledets splitting $T_\Gamma M = \bigoplus_{i=1}^{\ell} E^i$, 
 its Lyapunov exponents $\chi_1(\mu)>\cdots>\chi_\ell(\mu)$ from \eqref{e.oseledetselyapunov} and \eqref{e.expolyapunov}, and its Lyapunov spectrum
$L(\mu)=
	(\lambda_1(\mu),\dots, \lambda_m(\mu))$ from
\eqref{e.Lyapunovspectrummeasure}.
	
  The \emph{Lyapunov spectrum} of a measure $\mu\in \mathcal{M}_{\mathrm{inv}}(f)$
  with ergodic decomposition $\int \nu_\theta d \tau(\theta) $, where each $\nu(\theta)$ is an ergodic measure,
  is defined by
\[
L(\mu)\eqdef 
(\lambda_1(\mu),\dots, \lambda_m(\mu)), \,\,
\mbox{where} \,\,
\lambda_i(\mu)\eqdef 
\int \lambda_i(x) d\nu_\theta(x) d \tau(\theta)  \quad \text{for $i=1,\dots,m$,}
\]
where 
$\lambda_i(x)=\lambda_i(\nu_\theta)$ if $x$ is a generic point of $\nu_\theta$.

As in the case of ergodic measures, by definition, $\lambda_i(\mu) \geq \lambda_{i+1} (\mu)$.
The measure $\mu$ has \emph{simple  spectrum} if 
$\lambda_1(\mu)>\cdots> \lambda_m(\mu).$

We now determine the values $\lambda_i (\mu)$ using exterior products.
For   $i=0,1,\dots,m$ and $k\in\mathbb{N}_{\geq 1}$, define the function
\[
L_i^{(k)}(x,f)\eqdef \frac{1}{k}\log\|\wedge^i Df^k(x)\|,
\]
where $\wedge^i$ denotes exterior product of power $i$. Note that $L_i^{(0)}$ is the identity map.

\begin{remark}
\label{r.Lkiscont}
For every $k\in\mathbb{N}$,
the function $L_i^{(k)}(x,f)$ depends continuously on both $x\in M$ and $f\in\diff^1(M)$.
\end{remark}

Given $\mu\in\mathcal{M}_{\rm inv}(f)$, define
\[
L_i(\mu,f)\eqdef \liminf_{k\rightarrow \infty} \int L_i^{(k)}(x,f) d\mu(x).
\] 
Note that, as $(L_i^{(k)}(\cdot,f))_k$ is sub-additive, by Kingman's ergodic theorem, 
\[
L_i(\mu,f)\eqdef \lim_{k\rightarrow \infty} \int L_i^{(k)}(x,f) d\mu(x) 
= \inf_{k\in\mathbb{N}} \int L_i^{(k)}(x,f) d\mu(x).
\]

It holds, see for instance \cite[Section 3]{Rue:79},
\[
\lambda_i(\mu)=L_i(\mu,f)-L_{i-1}(\mu,f) \qquad  \text{for every $i=1,\dots,m$.} 
\]

\subsection{Convergence to ergodic measures: Axiomatized GIKN measures}
\label{ss.GIKN}
We now introduce
a (topological) criterion guaranteeing the convergence of a sequence of periodic measures to an ergodic measure.
Using this criterion, we introduce the
 so-called {\em{Axiomatized GIKN (ergodic) measures}} defined in \cite{KwiLac:}. 

Let us introduce some notation. In what follows,  given a finite set $X$, we denote by
 $\# X$ the number of its elements. Also, given a point $x$, the orbit of $x$ by $f$ is denoted by $\orb(f,x)$, or simply by $\orb(x)$ when $f$ is clear from the context. 
 
\begin{definition}[Good approximation] \label{Def:good-approximation}
	Given $p,q \in \per (f)$ and constants $0<\kappa\leq 1,\gamma>0$, we say $\orb(q)$ is a {\emph{$(\gamma,\kappa)$-good approximation}}  of $\orb(p)$ if there are a subset $X$ of $\orb(q)$ and a map $\rho\colon X\rightarrow \orb(p)$ such that:
	\begin{itemize}
		\item $\dfrac{\#X}{\pi(q)}>\kappa$;
		\item  $d(f^i(x),f^i(\rho(x)))<\gamma$ for every $x\in X$ and every $i\in\{0,1,\ldots,\pi(p)-1\}$;
		\item $\#\rho^{-1}(y)$ does not depend on $y\in\orb(p)$.
	\end{itemize}
\end{definition}

Note that the space  $\mathcal{M}_{\rm erg}(f)$ is in general noncompact
 and hence a sequence of ergodic measures may fail to converge to an ergodic one. The following criterion introduced in~\cite{Goretal:05}, see also~\cite{BonDiaGor:10}, 
guarantees the convergence to an ergodic measure of a sequence of periodic measures.

 \begin{lemma}[GIKN-criterion~{\cite{Goretal:05,BonDiaGor:10}}] \label{Lem:GIKN-criterion}
	Let $\{\orb(p_n)\}$ be a sequence of periodic orbits of $f$ and $\{\kappa_n\}$ and $\{\gamma_n\}$  sequences of real numbers with $\kappa_n\in (0,1]$ and $\gamma_n>0$, such that 
		\begin{itemize}
		\item 
		$\sum_{n=1}^\infty \gamma_n<+\infty$ and $\prod_{n=1}^{\infty} \kappa_n>0$,
		\item $\orb(p_{n+1})$ is a $(\gamma_n,\kappa_n)$-good approximation of $\orb(p_n)$ for every $n\geq 1$.
	\end{itemize} 
	Then the sequence of periodic measures  $\mu_{p_n}$ 
	weak$^\ast$ converges to an aperiodic measure $\mu\in \mathcal{M}_{\rm erg} (f)$ such that
	\[
	\supp(\mu)=\bigcap_{n\geq 1}\overline{\bigcup_{j\geq n} \orb(p_j)}.
	\]
\end{lemma}

\begin{definition}[Axiomatized GIKN measures] 
\label{def:GIKN}
We call a sequence of periodic measures satisfying
Lemma~\ref{Lem:GIKN-criterion} {\em{GIKN}} and its limit 
{\em{Axiomatized GIKN measure,}} or simply {\em{GIKN measure.}} 
\end{definition}

\begin{remark}[Properties of GIKN measures] 
GIKN measure are ergodic   \cite{Goretal:05} and have rank-one \cite{KwiLacTri:}
and thus they have zero entropy (see also \cite{KwiLac:}).
\end{remark}

\subsection{Homoclinic classes of $C^1$-generic diffeomorphisms}\label{ss.generic}

We now define the residual subset $\mathcal{R}$ of $\mathrm{Diff}^1(M)$ in the statement of our results. 

We call a {\em{saddle}} a hyperbolic periodic point whose stable and unstable sets are nontrivial. The {\em{index}} of a saddle is the dimension of its stable bundle.

\subsubsection{An auxiliary residual set of $\diff^1(M)$} 

We introduce a residual set $\mathcal{R}_0$ of $\diff^1(M)$ whose diffeomorphisms satisfy a series of ``good generic properties''. 
 
\begin{theorem}[$C^1$-generic properties of diffeomorphisms]\label{Thm:generic-properties}
There exists a residual subset $\mathcal{R}_0$ of $\diff^1(M)$ such that every $f\in\mathcal{R}_0$ and every nontrivial homoclinic class $H$ of $f$ the following properties hold:  
	\begin{enumerate}
		\item\label{generic:KS}  
		All periodic orbits  of $f$ are hyperbolic and their stable and unstable manifolds intersect transversely (which allows empty intersections).
		
		\item \label{saddleswithsimplespectrum}
		The saddles with a simple spectrum are dense in $H$.
		\item\label{generic:compact-convex} 
		The set $\mathcal{L}(H)$ 
		(recall \eqref{e.Lyapunovspectrumset})
		is compact and convex.
		\item\label{generic:new-periodic-orbit}
		Let $p\in H$ be a saddle with simple spectrum. Then for every $\varepsilon>0$, every $\kappa\in(0,1)$, and every $\gamma>0$, there exists  saddles $q\in H$ such that
		\begin{itemize}
			\item  the set $\orb(q)$ is $\varepsilon$-dense in $H$ and $q$ has simple spectrum;
			\item  $\orb(q)$ is homoclinically related to $\orb(p)$ and is a $(\gamma,\kappa)$-good approximation of $\orb(p)$;
			\item  $\left|\lambda_i(q)-\lambda_i(p)\right|<\varepsilon$ for every $i=1,\dots,m$.
		\end{itemize}
	\item\label{generic:homorelated}
	Every homoclinic class $H$ is a chain recurrence class\footnote{An {\em{\(\epsilon\)-pseudo orbit}} is a sequence of points $(x_i)$ such that $d(f(x_i), x_{i+1}) < \epsilon$.
The \emph{chain recurrent set} $\mathrm{CR}(f)$ of $f$ is the set of points $x \in M$ such that, for every $\epsilon > 0$, there exists an 
$\epsilon$-pseudo orbit that starts and ends at $x$.
Two points $x,y \in \mathrm{CR}(f)$ are in the same \emph{chain recurrence class} if, for every $\epsilon > 0$, there exists an $\epsilon$-pseudo orbit from $x$ to $y$ and another from $y$ to $x$. 
These classes are either equal or disjoint.} and depends continuously on the diffeomorphism. 
Moreover,
		every pair of saddles of $H$ with the same index are homoclinically related. 
		Moreover, given any pair of saddles $p_f, q_f$ of $f\in \mathcal{R}_0$ there is a neighborhood of 
		$\,\mathcal{V}_f$ of $f$
		where the continuations of these saddles are well defined and such that either their homoclinic classes coincide for every $g\in \mathcal{V}_f \cap \mathcal{R}_0$ or their homoclinic classes
		are disjoint for every $g\in \mathcal{V}_f \cap \mathcal{R}_0$. 
		 	\item
		\label{generic:lyapunovvector}
For every $\mu \in \mathcal{M}_{\mathrm{inv}} (H,f)$, 
there exists a sequence of saddles $\{\orb(\bar p_n)\}$ whose periodic measures 
$\mu_{\bar p_n}$ weak$^\ast$ converges to $\mu$,  $\orb(\bar p_n)$ converges to $\supp(\mu)$, and $L(\bar p_n)$ converges to $L(\mu)$. 
Note that this applies to the specific case of ergodic measures.
	\end{enumerate}
\end{theorem} 

The references for the assertions in Theorem~\ref{Thm:generic-properties} are the following:
Item~\ref{generic:KS} is  the classical Kupka-Smale Theorem~\cite{Kup:63,Sma:63}.
Item~\ref{saddleswithsimplespectrum} is formulated in \cite[Proposition 2.3]{Abdetal:07}.
Item~\ref{generic:compact-convex} is taken from~\cite[Corollary~2]{Abdetal:07}.
Item~\ref{generic:new-periodic-orbit} is taken from~\cite[Lemma~2.5]{Cheetal:19}.
In particular, in Item~\ref{generic:new-periodic-orbit}, the periodic measure $\mu_q$ can be chosen arbitrarily close to $\mu_p$ by taking $\kappa$ sufficiently close to $1$ and $\gamma$ sufficiently close to $0$.
Item~\ref{generic:homorelated} follows from \cite{CarMorPac:03}.
Item \ref{generic:lyapunovvector} follows from ~\cite[Theorem 3.5 (a)]{AbdBonCro:11} together with~\cite[Corollary 1.7]{BocBon:12}. Let us observe that \cite[Theorem 3.8]{AbdBonCro:11} formulates this item for ergodic measures.
Indeed, Item  \ref{generic:lyapunovvector} admits the following stronger version for isolated homoclinic classes.

 Note that 
for an isolated homoclinic class $H$, every
sequence of periodic measures $\{\mu_{q_n}\}$ that converges to some measure $\mu \in
\mathcal{M}_{\rm inv}(f)$ must satisfy that $\orb(q_n) \subset H$ for every large $n$.

\begin{lemma}\label{lProp:approximation-invariant-measure}
	Let $H$ be an isolated  nontrivial homoclinic class of $f\in \mathcal{R}_0$ and 
	 $\mu\in\mathcal{M}_{\rm inv}(H)$. Then there exists a sequence of periodic orbits $\{\orb(\bar p_n)\}$ contained in $H$ such that 
	\[
	\mu_{\bar p_n}
	\overset{\text{weak}\ast}{\longrightarrow} \mu \qquad \text{and} \qquad  L(\bar p_n)\rightarrow L(\mu) \quad \text{as $n\rightarrow\infty$.}
	\]
Moreover, the saddles $\bar  p_n$ can be chosen with a simple spectrum.
\end{lemma}

\begin{proof}
By Item \ref{generic:lyapunovvector} of Theorem~\ref{Thm:generic-properties}, there is a sequence of periodic orbits $\left\{\orb(\bar p_n)\right\}$ contained in $H$ such that $\{\mu_{\bar p_n}\}$ converges to $\mu$ and $L(\mu_{\bar p_n}) \to L(\mu)$.
Since $H$ is isolated, by the comment above, there is $n_0$ such that $\bar p_n \in H$ for all $n\ge n_0$. 

The fact that the measures can be chosen with a simple spectrum is a consequence of \cite[Lemma 4.9]{AbdBonCro:11}.
	 \end{proof}

\begin{remark}
\label{r.equality}
Item \ref{generic:lyapunovvector} in Theorem~\ref{Thm:generic-properties}
implies
the equality
$\mathcal{L}(H) = \mathcal{L}_{\mathrm{per}}(H)$ 
in \eqref{e.samedef} for every $f\in\mathcal{R}_0$ and every isolated  homoclinic class $H$.
\end{remark}

\subsubsection{Transitions in homoclinic classes}
We need to refine the residual set $\mathcal{R}_0$ in Theorem~\ref{Thm:generic-properties} to consider the existence of
periodic orbits in the class ``approximating''  simultaneously a pair of saddles. The ``difficult'' 
case occurs when
these saddles 
have different $\mathrm{s}$-indices. The results below
are
 motivated by the definition of a {\em{periodic linear system with transitions}} in
\cite[Definition 1.6]{BonDiaPuj:03}
 and its variation in \cite{Bonetal:03} for saddles having consecutive 
$\mathrm{s}$-indices, see Proposition~\ref{Prop:new-periodic-point} below.

We start considering saddles with the same $\mathrm{s}$-index.
We say that a $\mathrm{s}$-index $k$ is {\em{admissible for a homoclinic class}} if it contains saddles
of $\mathrm{s}$-index $k$. With these ingredients, 
we can reformulate \cite[Lemma 1.9]{BonDiaPuj:03}  for generic diffeomorphisms as follows:

\begin{lemma} \label{l.periodiclinearsystem}
Let $\mathcal{R}_0$ be the residual set in Theorem~\ref{Thm:generic-properties}.
Then for every $f\in \mathcal{R}_0$, every nontrivial homoclinic class of $f$, 
and every admissible $\mathrm{s}$-index $k$ of $H$, the derivative $Df$ induces a periodic linear system
with transitions on the set of
saddles of $H$ of  $\mathrm{s}$-index $k$.
\end{lemma}

We refrain from giving a precise definition of a periodic linear system with transitions. In Remark~\ref{r.periodiclineartranb},
we reformulate Lemma~\ref{l.periodiclinearsystem} in a more convenient way.

To consider saddles with different $\mathrm{s}$-indices, a preliminary 
step is to adapt Lemma~\ref{l.periodiclinearsystem} for saddles with a heterodimensional cycle (which is not a
generic condition).
Recall that two saddles $p,q$ of $f$ have a {\em{heterodimensional cycle}} if their invariant manifolds 
of their orbits
intersects cyclically:
$$
W^{\mathrm{s}}(\orb(p),f) \cap W^{\mathrm{u}}(\orb (q),f)\ne\emptyset\ne 
 W^{\mathrm{u}}(\orb(p),f) \cap W^{\mathrm{s}}(\orb (q),f).
 $$ 
  This concept will play an important role in our arguments.

In the sequel, we use the following notation for derivatives of periodic points $p$,
\[
M_p\eqdef D_pf^{\pi(p)}\colon T_pM\rightarrow T_pM. 
\]

We say that a pair of dominated splittings of two saddles $p$ and $q$
are {\em{compatible}} if they have the same number of bundles
	$$
	T_\ast M=E^1_\ast\oplus\cdots\oplus E^k_\ast, \qquad \ast \in \{\orb(p),\orb(q)\}
	$$
and moreover, it holds
	$$
	\dim(E^j_p)=\dim(E^j_q), \qquad \text{for every $j=1,\dots,k$}.
	$$
	 
In what follows, given a set $X$ and $\epsilon>0$, we let $B_{\epsilon}(X)$ the set of points at distance less than $\epsilon$ from $X$. 

\begin{proposition}[Theorem 3.1 in \cite{Bonetal:03}] \label{Prop:new-periodic-point}
Consider $f\in \diff^1(M)$ with a heterodimensional cycle associated with saddles $p=p_f$ and $q=q_f$
of $\mathrm{s}$-indices $i+1$ and $i$, respectively. 
Assume that $M_p$ and $M_q$ have compatible dominated splittings
	$$
	T_\ast M=E^1_\ast\oplus\cdots\oplus E^k_\ast, \quad \ast \in \{\orb(p),\orb(q)\}
	$$
such that	there is $r\in \{2, \dots, k-1\}$ such that $E^r$ has dimension one,
$E^1_p\oplus\cdots\oplus E_p^r= E^{\mathrm{s}}_p$, and $E_q^r \oplus \cdots E_q^k= E^{\mathrm{u}}_q$.

Then for every $\varepsilon, \gamma>0$  there are
 linear maps 
	$$
	T_0\colon T_pM\rightarrow T_qM \qquad \mbox{and} \qquad T_1\colon T_qM\rightarrow T_pM
	$$
	and 
	positive integers $t_0$ and $t_1$, depending only on $T_0$ and $T_1$, 
	such that
	for every $\alpha,\beta\in \mathbb{N}$  there is 
	a diffeomorphism $g$ $\varepsilon$-close to $f$
	having a saddle\,\footnote{This heavy notation will be needed in the genericity argument in Section~\ref{sss.genericity}.
	Similarly for the neighborhoods $\mathcal{O}_{f,\mathbf{v}}$ in Remarks~\ref{r.coroprop} and \ref{r.periodiclineartranb}.}
	\begin{equation}
	\label{e.pg}
 \bar p_{g, \mathbf{v}}, \qquad \mathbf{v}= (p,q, \gamma, \alpha, \beta),
	\end{equation}
	as follows:
	\begin{enumerate}
		\item\label{Prop:new-periodic-pointitem1} 
		$\pi(\bar p_{g, \mathbf{v}})=\alpha \pi(p)+\beta \pi(q)+t_0+t_1$;
		\item\label{Prop:new-periodic-pointitem2} 
		$\#(\orb(\bar p_{g, \mathbf{v}})\cap B_{\gamma}(\orb (p))\geq \alpha \pi(p)$ and 
		$\#(\orb(\bar p_{g, \mathbf{v}})\cap B_{\gamma}(\orb (q))\geq \beta \pi(q)$;
		\item\label{Prop:new-periodic-pointitem3}  
		there is a family of linear maps 
		$\{I_i\}_{i=0}^{\alpha+\beta+2}$  which are $(1/\ell)$-close to the identity, 
		\[
		\begin{split}
		&I_j \colon T_p M \to T_p M, \qquad \mbox{$j=0, \dots, \alpha$ and $j=\alpha+\beta+2$},\\
		&I_k \colon T_q M \to T_q M, \qquad k=\alpha+1, \dots, \alpha+\beta+1,
		\end{split}
		\]
		such that
		$M_{\bar p_{g}}$ is conjugate to the product
		\[ 
		I_{\alpha+\beta+2}\circ T_1\circ I_{\alpha+\beta+1} \circ M_q \circ I_{\alpha+\beta} \circ \cdots\circ I_{\alpha+2}\circ M_q \circ I_{\alpha+1}\circ T_0\circ I_\alpha\circ M_p\circ I_{\alpha-1}\circ  \cdots
		\circ I_1\circ M_p\circ I_0.
		\]	
	\end{enumerate}
	Moreover, $M_{\bar p_{g, \mathbf{v}}}$ has a dominated splitting 	
	compatible with the ones of
	$M_p$ and $M_q$ whose bundles are close (for the elements of the orbits of
	$\bar p_{g, \mathbf{v}}$ close to $\orb (p)$ and $\orb (q)$) to the ones of $\orb (p)$ and $\orb (q)$.
\end{proposition}

We observe that Proposition~\ref{Prop:new-periodic-point} is just \cite[Theorem 3.1]{Bonetal:03}, where we added a quantitative control of the visits to the neighborhoods of
 $\orb(p),\orb(q)$ in Item \ref{Prop:new-periodic-pointitem2} , which is implicit in the proof in \cite{Bonetal:03}.

\begin{remark}\label{r.thesaddlebarp}
By construction, the saddle $\bar p_{g,\mathbf{v}}$ in Proposition~\ref{Prop:new-periodic-point}
has well-defined strong stable and strong unstable manifolds which satisfy
\begin{equation}
\label{e.intersectionsstrong}
W^{\mathrm{uu}}(\orb (p_{g,\mathbf{v}})) \pitchfork W^{\mathrm{s}}(\orb(p_g)) \ne \emptyset,
\,\,
W^{\mathrm{ss}}(\orb (\bar p_{g,\mathbf{v}})) \pitchfork W^{\mathrm{u}}(\orb (q_g)) \ne \emptyset.
\end{equation}
This condition is implicit in the constructions in \cite{Bonetal:03} and is explicitly stated in
\cite[page 645, proof of Lemma 3.4]{WanZha:20}.

In particular, condition~\eqref{e.intersectionsstrong} implies that if $p_g$ and $q_g$ are in the same chain recurrence class, then the saddle $\bar p_{g,\mathbf{v}}$ also belongs to that class. Consequently, if the initial saddles $p_f$ and $q_f$ are locally generically in the same
homoclinic (or chain recurrence class), then $\bar p_{g,\mathbf{v}}$ is also in this class.
\end{remark}

Note also that the saddle $\bar p_{g, \mathbf{v}}$ has 
a continuation $\bar p_{h, \mathbf{v}}$ for diffeomorphisms $h$ close to $g$ satisfying Items \ref{Prop:new-periodic-pointitem1}--\ref{Prop:new-periodic-pointitem3} in Proposition~\ref{Prop:new-periodic-point}. This leads to the following:

\begin{remark}\label{r.coroprop}
Let $f\in \diff^1(M)$ have two saddles $p,q$ with a heterodimensional cycle as in Proposition~\ref{Prop:new-periodic-point}. Then 
for every $\gamma>0$ and  every $\alpha, \beta\in \mathbb{N}$ there is
an open
set $\mathcal{O}_{f, \mathbf{v}}$,
$\mathbf{v}=(p, q, \gamma, \alpha, \beta)$,
such that 
$f \in  \overline{\mathcal{O}_{f, \mathbf{v}}}$ 
and every $g\in \mathcal{O}_{f,\mathbf{v}}$ has
a saddle $\bar p_{g, \mathbf{v} }$ satisfying Items \ref{Prop:new-periodic-pointitem1}--\ref{Prop:new-periodic-pointitem3} of Proposition~\ref{Prop:new-periodic-point}.
\end{remark}

\begin{remark}[Same index version of Proposition~\ref{Prop:new-periodic-point}]
\label{r.periodiclineartranb}
The conclusions in Proposition~\ref{Prop:new-periodic-point} hold for  a diffeomorphism $f$ with saddles $p,q$ of the same 
$\mathrm{s}$-index  which are homoclinically related. In this case, it is enough to consider the corresponding
bundles $E^\mathrm{s}$ and $E^\mathrm{u}$ that have the same dimensions 
and no compatibility is required.
This fact is indeed a reformulation of Lemma~\ref{l.periodiclinearsystem}.
	As in Remark~\ref{r.coroprop}, we get open
sets $\mathcal{O}_{f,   \mathbf{v} }$, $\mathbf{v}=(p, q, \gamma, \alpha, \beta)$, with $f \in  \overline{\mathcal{O}_{f,  \mathbf{v} }}$, 
whose maps
$g$ have
 saddles $\bar p_{g,  \mathbf{v} }$ satisfying Items \ref{Prop:new-periodic-pointitem1}--\ref{Prop:new-periodic-pointitem3} of Proposition~\ref{Prop:new-periodic-point}.
	\end{remark}

\begin{remark}\label{r.whenHisisolated}
When $f \in \mathcal{R}_0$ and the saddles $p_f$ and $q_f$ belong to a robustly isolated homoclinic class $H$, then $\bar{p}_{g, \mathbf{v}}$ belongs to the  homoclinic class of these saddles. 
\end{remark}

\subsubsection{The residual set $\mathcal{R}$ of $\diff^1(M)$}\label{sss.genericity}
To define the residual set $\mathcal{R}$ in our results, we need a final ingredient:

\begin{lemma}[A consequence of the connecting lemma, \cite{Hay:97,WenXia:00}]
\label{l.hayashi}
Let $f\in \mathcal{R}_0$, 
$p_f,q_f$  be saddles of $f$ with different $\mathrm{s}$-indices, and 
 $\mathcal{V}_f$ a neighborhood of $f$
 such that the continuations of $p_f, q_f$ are defined 
 and  the
 homoclinic classes of $p_g,q_g$ are equal for every $g\in \mathcal{V}_f \cap \mathcal{R}_0$.
 Let $\mathcal{D}_f$ be the subset of $\mathcal{V}_f$ of diffeomorphisms  $g$ with a heterodimensional 
cycle associated to the
 $p_g$ and $q_g$. Then  $\mathcal{D}_f$ is dense in $\mathcal{V}_f$.
\end{lemma}

\begin{proof} We will proceed exactly as in \cite[Lemma 3.2]{DiaPujUre:99}. As the argument 
is short, we repeat it for completeness.
Assume that
the $\mathrm{s}$-index of $p_f$ is greater than the one of $q_f$. Take any $g \in \mathcal{R}_0\cap \mathcal{V}_f$. 
Since $p_g$ and $q_g$ are in the same transitive set
(a homoclinic class) a perturbation (provided by the connecting lemma in \cite{Hay:97, WenXia:00}) leads to $g_1$ such that $W^{\mathrm{s}}(p_{g_1},g_1)\cap W^{\mathrm{u}}(q_{g_1},g_1)\ne \emptyset$. 
Since, by hypotheses,  the sum of the dimensions of these manifolds is greater than $m$, 
after a perturbation, this intersection can be done transverse and hence robust, that is,
 $W^{\mathrm{s}}(p_{h},h)\pitchfork W^{\mathrm{u}}(q_h,h)\ne \emptyset$ for every $h$ in an open set whose closure contains $g_1$.
 After a new perturbation, we can assume that $h \in \mathcal{R}_0$ and hence
both saddles are in the same homoclinic class. Reversing now the roles of $p_h$ and $q_h$, we get $g_2$ close to 
$h$, hence to $g$,
such that
$W^{\mathrm{s}}(q_{g_2},g_2)\cap W^{\mathrm{u}}(p_{g_2},g_2)\ne \emptyset$. This provides a heterodimensional cycle and hence $g_2\in \mathcal{D}_f$. Since $g_2$ can be chosen arbitrarily close to
$g$, the result follows.
\end{proof}

We are now ready to construct the residual set $\mathcal{R}$. For each $n\in \mathbb{N}$, denote by 
$\mathrm{Per}_n (f)$ the set of periodic points of $f$
with period at most $n$. Let
$$
\mathcal{A}_n \eqdef \{ f\in \diff^1(M) \colon \mbox{every $p \in \mathrm{Per}_n (f)$ is hyperbolic}\}.
$$
The set $\mathcal{A}_n$ is open and dense in $\diff^1 (M)$ (indeed, this is part of the proof of the Kupka-Smale genericity theorem). 

For every
$f\in \mathcal{A}_n$, the set $\mathrm{Per}_n(f)$ is finite. List $p_{1,f}, \dots, p_{r,f}$ the saddles in $\mathrm{Per}_n (f)$ in such a way
their $\mathrm{s}$-indices are not decreasing. Consider the family $\mathcal{I}_{f,n}$ of pairs $(i,j)$ with
$1 \le i<j\le r$ such that the indices
of $p_{i,f}$ and $p_{j,f}$ are equal or differ by one. 
For each $(i,j)\in \mathcal{I}_{f,n}$ there is an open neighborhood $\mathcal{V}_{f,(i,j)}$
where the continuations of $p_{i,f}$ and $p_{j,f}$ are well defined and satisfy Item \ref{generic:homorelated} in
Theorem~\ref{Thm:generic-properties}. Consider the open neighborhood of $f$ given by
$$
\mathcal{V}_{f,n} \eqdef \bigcap_{(i,j)\in \mathcal{I}_{f,n}} \mathcal{V}_{f,(i,j)}.
$$

Consider the subset of $\mathcal{C}_{f,n}$ of $\mathcal{I}_{f,n}$ consisting of the pairs $(i,j)$ such that
 $p_{i,g}$ and $p_{j,g}$ have different (and hence consecutive) $\mathrm{s}$-indices and are in the same homoclinic class for every $g \in  \mathcal{V}_{f,n}\cap \mathcal{R}_0$. Let $\mathcal{D}_{f,(i,j)}$ be the dense subset  of $\mathcal{V}_{f,n}$ provided by Lemma~\ref{l.hayashi} applied
to $\mathcal{R}_0$, $\mathcal{V}_{f,n}$, and the saddles $p_{i,f}, p_{j,f}$. 
Select now $\ell, \alpha, \beta \in \mathbb{N}$.
In what follows, for notational simplicity, we identify the $5$-tuples
$$
(i, j, \frac1\ell, \alpha, \beta) \quad \mbox{and} \quad (p_{i,g}, p_{j,g}, \frac1\ell, \alpha, \beta).
$$
Applying Remark~\ref{r.coroprop}
 to $g\in  \mathcal{D}_{f,(i,j)}$ and their saddles $p_{i,g}, p_{j,g}$
with quantifiers $\gamma=\frac1\ell, \alpha, \beta$,
we get
an open set
$$
\mathcal{O}_{g, \mathbf{v}} \quad \mbox{with} \quad
g \in \overline{\mathcal{O}_{g,\mathbf{v}}}, \quad \mbox{where} \quad \mathbf{v}=(i, j, \frac1\ell, \alpha, \beta)
\equiv  (p_{i,g}, p_{j,g}, \frac1\ell, \alpha, \beta),
$$
consisting of diffeomorphisms $h$
having a saddle $\bar p_{h, \mathbf{v}}$ as in \eqref{e.pg}
 satisfying conditions (1)-(3) relative to
$p_{i,h}$ and $p_{j,h}$ for the quantifiers $1/\ell, \alpha, \beta$.
Write
$$
 \mathbf{v}=(i, j, \frac1\ell, \alpha, \beta)=(i,j, \mathbf{w}), \quad
\mbox{where} \quad 
 \mathbf{w}=(\frac1\ell, \alpha, \beta).
 $$
 Let
$$
\mathcal{W}_{f, n, \mathbf{w}} \eqdef \bigcap_{(i,j) \in \mathcal{C}_{f,n}} \mathcal{W}_{f,i,j,\mathbf{w}}, \quad
\mbox{where} \quad
\mathcal{W}_{f, i,j, \mathbf{w}} \eqdef \bigcup_{g \in \mathcal{D}_{(i,j),f}}  \mathcal{O}_{g, i,j ,\mathbf{w}}.
$$
By construction,  the set $\mathcal{W}_{f, n,\mathbf{w}}$  is open and dense in $\mathcal{V}_{f,n}$.
Consider the set
$$
\mathcal{V}_{\mathcal{C}, n, \mathbf{w}} \eqdef \bigcup_{f\in \mathcal{R}_0} \mathcal{W}_{f, n, \mathbf{w}},
$$
which is open and dense in $\diff^1(M)$.

We now consider the subset of $\mathcal{H}_{f,n}$ of $\mathcal{I}_{f,n}$ consisting of the pairs $(i,j)$ such that
 $p_{i,g}$ and $p_{j,g}$ have the same  $\mathrm{s}$-index and are in the same homoclinic class for every $g \in  \mathcal{V}_{f,n}\cap \mathcal{R}_0$. Arguing as above, now considering Remark~\ref{r.periodiclineartranb}
 and the corresponding open sets and notation, we define
 $$
\mathcal{U}_{f, n, \mathbf{w}} \eqdef \bigcap_{(i,j) \in \mathcal{H}_{f,n}} \mathcal{U}_{f,i,j,\mathbf{w}}, \quad
\mbox{where} \quad
\mathcal{U}_{f, i,j, \mathbf{w}} \eqdef \bigcup_{g \in \mathcal{R}_0}  \mathcal{O}_{g, i,j \mathbf{w}}.
$$
By construction,  the set $\mathcal{U}_{f, n, \mathbf{w}}$  is open and dense in $\mathcal{V}_{f,n}$.
Consider the set
$$
\mathcal{V}_{\mathcal{H}, n, \mathbf{w}} \eqdef \bigcup_{f\in \mathcal{R}_0} \mathcal{U}_{f, n, \mathbf{w} },
$$ 
which is open and dense in $\diff^1(M)$.

Therefore, by construction, letting $\overline{\mathbf{w}}=(\ell, \alpha, \beta)\in \mathbb{N}^3$
for $\mathbf{w} =(1/\ell, \alpha, \beta)$,
the set
\begin{equation}
\label{e.theserR1}
\mathcal{R}_1 \eqdef 
 \bigcap_{ \ast \in \{\mathcal{C}, \mathcal{H}\}, \, n \in \mathbb{N}, \, \overline{\mathbf{w}} \in \mathbb{N}^3} \,\, \mathcal{V}_{\ast, n, \mathbf{w}} 
\end{equation}
is residual in $\diff^1(M)$.
We finally let, 
\begin{equation}
\label{e.thesetR}
\mathcal{R}\eqdef  \mathcal{R}_0 \cap \mathcal{R}_1, \qquad \mbox{$\mathcal{R}_0$ as in Theorem \ref{Thm:generic-properties} and $\mathcal{R}_1$ as in \eqref{e.theserR1}.}
\end{equation}
which is by definition a residual set of $\diff^1(M)$.

\begin{remark}[Key property of $\mathcal{R}$]
\label{r.keyresidual}
Given any $f\in \mathcal{R}$, any pair of saddles $p_f,q_f$
in the same homoclinic class 
and having the same or consecutive $\mathrm{s}$-indices, and natural numbers $\ell, \alpha, \beta$, 
there is a saddle 
$$
\bar p_{f, \mathbf{v}}, \qquad
\mathbf{v}=p_f,q_f,1/\ell, \alpha, \beta,
$$
as in 
\eqref{e.pg} satisfying Items \ref{Prop:new-periodic-pointitem1}--\ref{Prop:new-periodic-pointitem3} of Proposition ~\ref{Prop:new-periodic-point}
relative to $\ell, \alpha, \beta$ and the pair of saddles $p_f, q_f$. Moreover, if $p_f$ and $q_f$ are in the same chain recurrence class
the saddle $\bar p_{f, \mathbf{v}}$ is also in this class, see Remark~\ref{r.thesaddlebarp}.
\end{remark}

\section{Proof of Theorem~\ref{Thm:Lyapunov-spectrum}}
\label{s.proofofThm:Lyapunov-spectrum}

In this section, we prove Theorem~\ref{Thm:Lyapunov-spectrum} that is derived from 
the approximation Theorem~\ref{Thm:construction-sequences}, where a GIKN sequence of measures with good
properties is constructed. 
The proof of this result is given in Section~\ref{ss.construction-sequences}. 
The proof of Theorem~\ref{Thm:Lyapunov-spectrum} is completed in Section~\ref{ss.endThm:Lyapunov-spectrum}.

In what follows,  we use the box metric on $\mathbb{R}^m$: the box distance between two vectors
$u=(u_1,\dots,u_m), w=(w_1,\dots,w_m)$, 
is
\[
d(u,w)\eqdef \|u-w\| \eqdef \max_{1\leq i\leq m}\big\{|u_i-w_i|\big\}.
\]

\begin{theorem}\label{Thm:construction-sequences}
	Let $\mathcal{R}$ be the residual subset of $\diff^1(M)$ in \eqref{e.thesetR}.
	Then for every $f\in\mathcal{R}$, every 
	nontrivial homoclinic class $H$ of $f$, and every $L\in \interior (\mathcal{L}_{\mathrm{per}}(H))$
	there exist $\delta\in (0,1)$  and sequences of periodic points $\{p_n\}_{n\in \mathbb{N}_0}$ in $H$,
	of  constants $\left\{\gamma_n\right\}_{n\in \mathbb{N}_0}$, and of positive integers $\left\{k_n\right\}_{n\in \mathbb{N}_0}$,
	with  
	$$
	0<\gamma_{n+1}<\frac{\gamma_n}{2}<\frac{1}{2} \qquad \text{and} \qquad k_{n+1} > k_n,
	$$
such that
	for every $n\in\mathbb{N}$ the following holds:
	\begin{enumerate}
	
		\item\label{item:densness} $\orb(p_n)$ has simple spectrum and is $\frac{\delta}{2^n}$-dense in $H$;
		
		\item\label{item:good-approximation} 
		$\orb(p_{n+1})$ is a $(\gamma_n,1-\frac{1}{2^n})$-good approximation of $\orb(p_n)$;
		
		\item\label{item:Lyapunov-spectrum} let $L=(\lambda_1,\dots,\lambda_m)$, then  
		\[
		\lambda_i-\frac{\delta}{2^n}<\lambda_i(p_n)<\lambda_i+\frac{\delta}{2^{n}} \quad \text{for every $i=1,\dots,m$,}\] 
		and thus $\left\|L(p_n)-L\right\|<\frac{\delta}{2^n}$;
		
		\item\label{item:exterior-product} 
		for every $n\in\mathbb{N}$ and every $x\in B_{2\gamma_n}(\orb(p_n))$, it holds
		\[\left|L_i^{(k_n)}(x,f)-L_{i-1}^{(k_n)}(x,f)-\lambda_i(p_n)\right|<\frac{\delta}{2^n}
		\qquad \text{for every $i=1,\dots,k$.}
		\]

	\end{enumerate}
\end{theorem}

\subsection{Proof of Theorem~\ref{Thm:construction-sequences}}
\label{ss.construction-sequences}
	Take  the residual subset  $\mathcal{R}\subset\diff^1(M)$  in 
	\eqref{e.thesetR}. Let
	$f\in \mathcal{R}$,
	 $H$ be a nontrivial homoclinic class of $f$, and $L=(\lambda_1,\dots,\lambda_k)\in\interior (\mathcal{L}_{\mathrm{per}}(H))$.

	\medskip
	
	\noindent{\bf{The constant $\delta$.}}
	Take $0<\delta<1$ such that 
	\begin{equation}
	\label{e.ballcontained}
	B_{4\delta}(L)
	\subset \interior \mathcal{L}(H).
	\end{equation}
		
	We now define inductively the sequences $\left\{\orb(p_n)\right\}$, $\left\{\gamma_n\right\}$, and  
	$\left\{k_n\right\}$.
	
	\medskip
	
	\noindent{\bf{The periodic point $p_0$, the constants $\gamma_0$ and $k_0$.}}
	Since $L\in\interior ( \mathcal{L}_{\mathrm{per}}(H))$, by Item~\ref{generic:new-periodic-orbit} of Theorem~\ref{Thm:generic-properties}, there exists a hyperbolic periodic point $p_0\in H$ such that
	\begin{itemize}
		\item $\orb(p_0)$ has simple spectrum and is $\delta$-dense in $H$, which is Item~\ref{item:densness}
		of Theorem~\ref{Thm:generic-properties};

		\item  $\lambda_i-\delta<\lambda_i(p_0)<\lambda_i+\delta$ for each $i=1,2,\ldots,m$, 
		thus $\big\|\lambda_i(p_0)-\lambda_i\big\|<\delta$, which is Item~\ref{item:Lyapunov-spectrum}
		of Theorem~\ref{Thm:generic-properties}.
	\end{itemize}
	Since $p_0$ is a periodic point of $f$ with simple spectrum, for every $x\in\orb(p_0)$ and every $i=1,\dots,m$, it holds
	\[
	L_i(p_0,f)=\lim_{k\rightarrow \infty} L_i^{(k)}(x,f) \qquad \text{and} \qquad \lambda_i(p_0)=\lim_{k\rightarrow \infty} \left(L_i^{(k)}(x,f)-L_{i-1}^{(k)}(x,f)\right).
	\]
	Thus there is $k_0\in\mathbb{N}$ such that for every $k\geq k_0$,  every $x\in\orb(p_0)$, and every $i=1,\dots,m$, it holds
	\[
	\left|L_i^{(k)}(x,f)-L_{i-1}^{(k)}(x,f)-\lambda_i(p_0)\right|<\delta.
	\]
	By Remark~\ref{r.Lkiscont}, the 
	functions $L_i^{(k_0)}(x,f), i=1,\dots,m$, are continuous with respect to $x$ and $f$. Thus there exists $0<\gamma_0<1/2$ such that for every $x\in B_{2 \gamma_0}(\orb(p_0))$ and every $i=1,\dots,m$,
	\[\left|L_i^{(k_0)}(x,f)-L_{i-1}^{(k_0)}(x,f)-\lambda_i(p_0)\right|<\delta,\]
	obtaining Item~\ref{item:exterior-product}.
	
	Note  that, at this step of the inductive proof, we do  not yet  need to verify Item~\ref{item:good-approximation}  for
	 $\orb(p_0)$. This
	completes the choice of the first periodic point $p_0$.

\medskip	
\noindent{\bf{Inductive definition of  $\orb(p_{n+1}),\gamma_{n+1},k_{n+1}$.}}
	Assume that $\orb(p_i),\gamma_i,k_i$ have been constructed for all $0\leq i\leq n$ satisfying Items~\ref{item:densness},~\ref{item:good-approximation},~\ref{item:Lyapunov-spectrum}, and~\ref{item:exterior-product}.
	In particular, $\orb(p_n)$ has simple spectrum and 
	\[
	\lambda_i-\frac{\delta}{2^n}<\lambda_i(p_n)<\lambda_i+\frac{\delta}{2^{n}} \qquad \text{for every $i=1,\ldots,m$}.
	\]

	Before going into the details, we sketch the construction of $\orb(p_{n+1})$. First, we choose a hyperbolic periodic orbit $\orb(q_n)$ with simple spectrum that is ``very dense''  in $H$ and such that $i$th-coordinates of $L(q_n)$  and $L(p_n)$ in $B_{2\delta}(L)$ are almost antipodal, where $L$ is the target vector. Applying Proposition~\ref{Prop:new-periodic-point} or its same index version in
Remark~\ref{r.periodiclineartranb},
	 we then obtain a new periodic orbit $\orb(p_{n+1})$ that closely approximates $\orb(p_n)$ and spends some time in a small neighborhood of $\orb(q_n)$, so that the distance from $L(p_{n+1})$ to $L$ is half the distance from $L(p_n)$ to $L$. A careful analysis shows that such an orbit $\orb(p_{n+1})$ indeed exists. We now go into the details.

	\medskip
	\noindent{\bf{The auxiliary orbit $\orb(q_n)$, chosen to adjust  the Lyapunov spectrum.} }
	Recall that the ball $B_{4\delta}(L)$ is contained in $\mathcal{L}_{\mathrm{per}}(H)$,
	\eqref{e.ballcontained},
	thus every element in $B_{4\delta}(L)$ can be approximated by the Lyapunov spectrum of some 
	periodic orbits contained in $H$, recall
	Remark~\ref{r.equality}.

	\begin{remark}[Choice of the point $q_n$]
	\label{rem:choiceofqn}
	By Item~\ref{generic:new-periodic-orbit} of Theorem~\ref{Thm:generic-properties} applied to $p_n$, there is a saddle $q_n\in H$  such that
	\begin{itemize}
		\item[(a)] $\orb(q_n)$ has simple spectrum and is $\frac{\delta}{4^{n+1}}$-dense in $H$;
		
		\item[(b)] for each $i=1,\dots,m$, the $i$th Lyapunov exponent of $\orb(q_n)$ satisfies:
		\begin{itemize}
			\item if $\lambda_i(p_n)\in (\lambda_i-\frac{\delta}{2^n},\lambda_i]$, then 
			\[
			\lambda_i+2\delta-\frac{\delta}{4^{n+1}}<\lambda_i(q_n)<\lambda_i+2\delta+\dfrac{\delta}{4^{n+1}};
			\]
			\item if $\lambda_i(p_n)\in (\lambda_i,\lambda_i+\frac{\delta}{2^n})$, then 
			\[
			\lambda_i-2\delta-\frac{\delta}{4^{n+1}}<\lambda_i(q_n)<\lambda_i-2\delta+\frac{\delta}{4^{n+1}}.
			\] %$\lambda_i+2\delta-\frac{\delta}{4^{n+1}}<\lambda_i(q_n)<\lambda_i+2\delta+\frac{\delta}{4^{n+1}}$ ;
		\end{itemize} 
		\end{itemize}
	\end{remark}
	
	\begin{remark} When $\lambda_i \neq 0$, the saddles $p_n$ and $q_n$ have the same $\mathrm{s}$-index. 
When $\lambda_i = 0$, the $\mathrm{s}$-indices of the saddles $p_n$ and $q_n$ are consecutive.
In both cases, 
Remark~\ref{r.keyresidual} about the set $\mathcal{R}$ allows the definition of the next generation saddle $p_{n+1}$.
This will be done in \eqref{e.saddlegenerationn+1}.
\end{remark}

	Since both $\orb(p_n)$ and $\orb(q_n)$ have simple spectrum, the tangent spaces $T_{\orb(p_n)}M$ and 
	$T_{\orb(q_n)}M$ admit dominated splittings (with respect $Df$) with one-dimensional bundles, say
\[
	T_{\orb(p_n) \cup \orb (q_n)}M=E^1\oplus\cdots\oplus E^m.
\]
	Furthermore,  
	taking $r=p,q$,
	for each $i=1,\dots,m$, one has, 
		\begin{equation}
		\label{Equ:pqn}
		\lambda_i(r_n)=\frac{\log\left\|Df^{\pi(r_n)}|_{E^i_{x}}\right\|}{\pi(r_n)}=L_i^{\pi(r_n)}(x,f)-L_{i-1}^{\pi(r_n)}(x,f)\quad \text{for all $x\in\orb(r_n)$}.
\end{equation}

\medskip

	\noindent{\bf{The candidate periodic orbit $\orb(p_{n+1})$.}} 
	Given any $\varepsilon, \gamma>0$ (the number $\gamma$ to be fixed later),
	having in mind the definition of the residual set $\mathcal{R}$, considering
	$p_n$ and $q_n$  and arbitrarily large numbers 
	 $\alpha,\beta\in \mathbb{N}$ we get a  saddle 
	 \begin{equation}
	 \label{e.saddlegenerationn+1}
	 p_{n+1}\eqdef \bar p_{f,p,q,\gamma, \alpha, \beta}
	 \quad 
	 \mbox{as in \eqref{e.pg}}
	 \end{equation}
	 satisfying Items \ref{Prop:new-periodic-pointitem1}--\ref{Prop:new-periodic-pointitem3} in Proposition~\ref{Prop:new-periodic-point}.
	 In particular,
  \begin{itemize}[leftmargin=0.4cm ]
		\item $\pi(p_{n+1})=\alpha \pi(p_n)+\beta \pi(q_n)+t_0+t_1$;
		\item 
		$\#(\orb(p_{n+1})\cap 
	B_{\gamma_{n+1}}(\orb(p_n))
		\geq \alpha \pi(p_n)$ and $\#(\orb(p_{n+1})\cap 
		B_{\gamma_{n+1}}(\orb(q_n))
		\geq \beta \pi(q_n)$; 
		\item 
		$M_{p_{n+1}}$ 
 admits a dominated splitting 
 compatible with the ones of $M_{p_n}, M_{q_n}$.
 Thus
 $T_{p_{n+1}}M=E^1_{p_{n+1}}\oplus\cdots\oplus E^m_{p_{n+1}}$ with
 one-dimensional bundles.
 Therefore $\orb(p_{n+1})$ has simple spectrum.
	\end{itemize}

Finally, since the saddles $p_n$ and $q_n$ are in the same 
homoclinic class $H$, which is a 
chain recurrence class, by Item~\ref{generic:homorelated} of Theorem~\ref{Thm:generic-properties},
 Remark~\ref{r.thesaddlebarp}
implies that $p_{n+1} \in H$.
	
	In the following, we have to choose a sufficiently small $\gamma$ and large
	integers $\alpha,\beta\gg 1$ to guarantee that the periodic orbit $\orb(p_{n+1})$ satisfies  
	all items in the theorem.
	
	\medskip
	\noindent{\bf{Choice of $\gamma$ to verify Item~\ref{item:densness}.} }Without loss of generality and by construction, we can assume that 
	\[
	p_{n+1}\in B_\gamma(p_n) \qquad \text{and}\qquad  f^{t_0}(p_{n+1})\in B_\gamma (q_n). 
	\]
If $\gamma$ is small enough, then
	$E^i_{p_{n+1}}$ is $C^0$-close to $E^i_{p_{n}}$ and $E^i_{f^{t_0}(p_{n+1})}$ is $C^0$-close to $E^i_{q_{n}}$ for each $i=1,\ldots,m$. More precisely,
	take 
	\begin{equation}
	\label{e.choiceofgamma}
	0<\gamma<\min\left\{\frac{1}{4^{n+1}},\frac{\gamma_n}{2} \right\}
	\end{equation}	
	 such that the following estimates hold:
  \begin{itemize}[leftmargin=0.6cm ]
		\item If $x\in B_\gamma (\orb(p_n))$ then
		$f^j(x)\in B_{\gamma_n}(\orb(p_n))$ for every $j=1,\dots, \pi(p_n)$,
		
		\item If  $y\in B_\gamma (\orb(q_n))$, then $f^j(x)\in B_{\gamma_n}(\orb(q_n))$ for every $j=1,\dots, \pi(q_n)$,

		\item for $r\in \{p,q\}$ and every $i=1,\dots,m$ it holds
			
		\begin{equation}\label{Equ:pqn+1-1}
			\left|\log\left\|Df^{\pi(r_{n})}|_{E^i_{r_{n+1}}}\right\|-\log\left\|Df^{\pi(r_n)}|_{E^i_{r_{n}}}\right\|\right|<
			\pi(r_n) \delta_n, \qquad \delta_n \eqdef \frac{\delta}{1000\cdot 4^n}.
		\end{equation}
		\end{itemize}

    Recall that the auxiliary orbit   $\orb(q_n)$ is $\frac{\delta}{4^{n+1}}$-dense in $H$ and satisfies
    \[
    \#\big(\orb(p_{n+1})\cap B_\gamma (\orb(q_n))\big)
    \geq \beta\cdot\pi(q_n).
    \]
    Thus, by the choice of $\gamma$,  $\orb(p_{n+1})$ is $2 \frac{\delta}{4^{n+1}}$-dense in $H$ provided $\beta>1$. Thus $\orb(p_{n+1})$ is  $\frac{\delta}{2^{n+1}}$-dense in $H$ and Item~\ref{item:densness} holds.
    
    \medskip
	\noindent{\bf{Choices of $\alpha,\beta$ to verify Items~\ref{item:good-approximation} and~\ref{item:Lyapunov-spectrum}.}}
	
	\begin{claim}
	\label{cl.goodapprox}
	$\orb(p_{n+1})$ is a
	$(\gamma_n,\frac{\alpha \pi(p_n)}{\pi(p_{n+1})})$-good approximation 
	of $\orb(p_n)$.
	\end{claim}
	
	\begin{proof} Note that
	$$
	\pi(p_{n+1})=\alpha\cdot \pi(p_n)+\beta \pi(q_n)+t_0+t_1
	\quad \text{and} \quad
	\#(\orb(p_{n+1})\cap B_\gamma (\orb(p_n)))\geq \alpha \pi(p_n).
	$$
	 As
	 $\gamma\in (0,\gamma_n)$,
	the set $\orb(p_{n+1})$ is a 
	$(\gamma,\frac{\alpha\cdot \pi(p_n)}{\pi(p_{n+1})})$-good approximation of $\orb(p_n)$, implying the claim.
	\end{proof}
	
Recall the dominated splitting $T_{p_{n+1}}M=E^1_{p_{n+1}}\oplus\cdots\oplus E^m_{p_{n+1}}$ with one-di\-men\-sional bundles of  $M_{p_{n+1}}=D_{p_{n+1}} f^{\pi(p_{n+1})}$. Thus the $i$th-Lyapunov exponent of $\orb(p_{n+1})$ satisfies:
	\begin{align*}
		\lambda_{i}(p_{n+1})=\frac{1}{\pi(p_{n+1})}\log \left\|Df^{\pi(p_{n+1})}|_{E^i_{p_{n+1}}}\right\|=
	\frac{1}{\pi(p_{n+1})}\sum_{j=0}^{\pi(p_{n+1})-1} \log \left\|Df|_{E^i_{f^j(p_{n+1})}}\right\|.
      \end{align*}
    Thus, by~\eqref{Equ:pqn} and \eqref{Equ:pqn+1-1}, the following estimates hold:
    \begin{gather*}
    	\lambda_i(p_{n+1})<\frac{\alpha \pi(p_n) \left(\lambda_i(p_n)+\delta_n\right)+\beta \pi(q_n)\cdot\left(\lambda_i(q_n)+\delta_n \right)+(t_0+t_1)\log\|Df\|}{\alpha \pi(p_n)+\beta \pi(q_n)+t_0+t_1},\\
    	\lambda_i(p_{n+1})>\frac{\alpha \pi(p_n)\left(\lambda_i(p_n)-\delta_n\right)+\beta \pi(q_n)\cdot\left(\lambda_i(q_n)-\delta_n \right)+(t_0+t_1)\log m(Df)}{\alpha \pi(p_n)+\beta \pi(q_n)+t_0+t_1},
    \end{gather*}
    where $m(Df)$ is the co-norm of $Df$. To simplify notation, write
    $$ 
    K_1\eqdef  (t_0+t_1)\log\|Df\|, \qquad \mbox{and} \qquad K_2 \eqdef (t_0+t_1)\log m(Df).    
    $$
    
Since $t_0,t_1$ are fixed constants that only depend on $B_\gamma (\orb(p_n))$ and $B_\gamma(\orb(q_n))$ and the linear maps $T_0,T_1$, to guarantee that $\orb(p_{n+1})$ satisfies Items~\ref{item:good-approximation} and~\ref{item:Lyapunov-spectrum}, it is enough to take $\alpha$ and $\beta$ large enough  satisfying the following inequalities:

   \begin{equation}
    \begin{split}\label{Equ:good-approximation-all}
    	\frac{\alpha \pi(p_n)}{\alpha \pi(p_n)+\beta \pi(q_n)+t_0+t_1}&>1-\frac{1}{2^n},\\
     	\frac{\alpha \pi(p_n) \left(\lambda_i(p_n)+\delta_n\right)+\beta \pi(q_n) \left(\lambda_i(q_n)+\delta_n \right)+ K_1}{\alpha \pi(p_n)+\beta \pi(q_n)+t_0+t_1}&<\lambda_i+\frac{\delta}{2^{n+1}},\\
    	\frac{\alpha \pi(p_n) \left(\lambda_i(p_n)-\delta_n\right)+\beta \pi(q_n) \left(\lambda_i(q_n)-\delta_n\right)+K_2}{\alpha \pi(p_n)+\beta \pi(q_n)+t_0+t_1}&>\lambda_i-\frac{\delta}{2^{n+1}}.
    \end{split}
   \end{equation}
   
  We now see how to get all these inequalities simultaneously.    Denote by 
    $$
    t(\alpha, \beta)=\frac{\alpha \pi(p_n)}{\alpha \pi(p_n)+\beta \pi(q_n)}.
    $$
    Taking $\alpha,\beta\rightarrow\infty$
    such that the limit $t(\alpha, \beta)$ exists and is equal to $t$,
    the left hand of the three inequalities in \eqref{Equ:good-approximation-all} converge  respectively to 
    \begin{align}\label{Equ:limit}
    t,\qquad  t\lambda_i(p_n)+(1-t)\lambda_i(q_n)+\delta_n,\qquad t\lambda_i(p_n)+(1-t)\lambda_i(q_n)-\delta_n.
    \end{align}
    Recalling the choices of 	$\lambda_i(p_n)$ and $\lambda_i(q_n)$  in Remark~\ref{rem:choiceofqn} we get:
  \begin{itemize}[leftmargin=0.6cm ]
    	\item[--] if $\lambda_i(p_n)\in (\lambda_i-\frac{\delta}{2^n},\lambda_i]$, then $\lambda_i+2\delta-\frac{\delta}{4^{n+1}}<\lambda_i(q_n)<\lambda_i+2\delta+\frac{\delta}{4^{n+1}}$, thus 
    	\[\begin{split}
	t(\lambda_i-\frac{\delta}{2^{n}})+(1-t)(\lambda_i+2\delta-\frac{\delta}{4^{n+1}})
	&<t\lambda_i(p_n)+(1-t)\lambda_i(q_n)\\
	&<t\lambda_i+(1-t)(\lambda_i+2\delta+\frac{\delta}{4^{n+1}});\end{split}\]
    	
    	\item[--] if $\lambda_i(p_n)\in (\lambda_i,\lambda_i+\frac{\delta}{2^n})$, then $\lambda_i-2\delta-\frac{\delta}{4^{n+1}}<\lambda_i(q_n)<\lambda_i-2\delta+\frac{\delta}{4^{n+1}}$, thus 
    	    	\[\begin{split}
t\lambda_i+(1-t)(\lambda_i-2\delta-\frac{\delta}{4^{n+1}})
	&<t(\lambda_i+\frac{\delta}{2^n})+(1-t)\lambda_i(q_n)\\
	&<t(\lambda_i+\frac{\delta}{2^n})+(1-t)(\lambda_i-2\delta+\frac{\delta}{4^{n+1}}).
	\end{split}\]
    \end{itemize} 
    Therefore to get simultaneously all the inequalities in \eqref{Equ:good-approximation-all}, it is enough to choose $\alpha,\beta$ such  that the limits in \eqref{Equ:limit} satisfy
     \begin{equation}
     \begin{split}\label{Equ:sufficient-condition}
     		t&>1-\frac{1}{2^n},\\
     		t(\lambda_i-\frac{\delta}{2^{n}})+(1-t)(\lambda_i+2\delta-\frac{\delta}{4^{n+1}})&> \lambda_i-\frac{\delta}{2^{n+1}}+\delta_n,  \\
     		t\lambda_i+(1-t)(\lambda_i-2\delta-\frac{\delta}{4^{n+1}})&>\lambda_i-\frac{\delta}{2^{n+1}}+\delta_n,\\
     		t\lambda_i+(1-t)(\lambda_i+2\delta+\frac{\delta}{4^{n+1}})&<\lambda_i+\frac{\delta}{2^{n+1}}-\delta_n,\\
     		t(\lambda_i+\frac{\delta}{2^n})+(1-t)(\lambda_i-2\delta+\frac{\delta}{4^{n+1}})&<\lambda_i+\frac{\delta}{2^{n+1}}-\delta_n.
 	    \end{split}
      \end{equation}
     To proceed, let now
       \begin{equation}
       \label{e.varrho}
          \varrho_n^+ \eqdef \frac{\frac{1}{2^{n+1}}-\frac{1}{1000\cdot 4^n}}{2+\frac{1}{4^{n+1}}},
          \qquad 
          \varrho_n^- \eqdef\frac{\frac{1}{2^{n+1}}+\frac{1}{1000\cdot 4^n}}{2+\frac{1}{2^n}-\frac{1}{4^{n+1}}},
      \end{equation}
      and for future use observe that
      \begin{equation}
      \label{e.lastcall}
      \frac{1}{2^n} < \varrho^-_n < \varrho_n^+.
      \end{equation}
      Thus the interval $(1-\varrho_n^+, 1-\varrho_n^-)$ is nonempty.
      Note that the 
      inequalities in \eqref{Equ:sufficient-condition} 
      are equivalent to
      $$
      1-\varrho_n^+ < t < 1-\varrho_n^-.
      $$
      As a consequence, one can choose 
      infinitely many pairs
       $\alpha,\beta\in\mathbb{N}$ such that 
          \begin{align}\label{Equ:choice-t}
      	  t(\alpha, \beta) = \frac{\alpha\cdot \pi(p_n)}{\alpha\cdot \pi(p_n)+\beta\cdot\pi(q_n)}\in 
	  (1-\varrho_n^+, 1-\varrho_n^-).
      \end{align}

      For these choices, the inequalities in 
      \eqref{Equ:sufficient-condition} hold (hence  
       \eqref{Equ:good-approximation-all}
      also holds). Thus $\orb(p_{n+1})$ satisfies conditions in Items~\ref{item:good-approximation} 
      (where we use \eqref{Equ:choice-t} and \eqref{e.lastcall})
      and~\ref{item:Lyapunov-spectrum}.
      
      \medskip
      
      \noindent{\bf{Choice of the constants $k_{n+1},\gamma_{n+1}$.}}
      This choice is analogous to the one of $k_0,\gamma_0$. We explicit it for completeness. The periodicity 
      and the simple spectrum property 
      of $p_{n+1}$  imply that for every $x\in\orb(p_{n+1})$ and every $i=1\dots,m$,
      \[
      L_i(p_{n+1},f)=\lim_{k\rightarrow \infty} L_i^{(k)}(x,f) \qquad \text{and} \qquad \lambda_i(p_{n+1})=\lim_{k\rightarrow \infty} \left(L_i^{(k)}(x,f)-L_{i-1}^{(k)}(x,f)\right).\]
      Thus there is $k_{n+1}\in\mathbb{N}$, with $k_{n+1}>k_n$, such that 
      for every 
      $x\in\orb(p_{n+1})$ it holds
      \[
      \left|L_i^{(k_{n+1})}(x,f)-L_{i-1}^{(k_{n+1})}(x,f)-\lambda_i(p_{n+1})\right|<\frac{\delta}{2 \cdot 4^n}
      \qquad \text{for every $i=1,\dots,m$}.
      \]
      Since the function $L_i^{(k_{n+1})}(x,f)$ is continuous with respect to $x$ and $f$, recall Remark~\ref{r.Lkiscont},
       there exists $0<\gamma_{n+1}<\frac{\gamma_n}{2}$ such that for any $y\in B_{2\gamma_{n+1}}(\orb(p_{n+1}))$ and every $i=1,\dots,m$, it holds
      \[
      \left|L_i^{(k_{n+1})}(y,f)-L_{i-1}^{(k_{n+1})}(y,f)-\lambda_i(p_{n+1})\right|<\frac{\delta}{4^n}.
      \]
    Therefore, Item~\ref{item:exterior-product} holds for $\orb(p_{n+1})$.
          This ends the proof of Theorem~\ref{Thm:construction-sequences}.
\qed

\subsection{End of the proof of Theorem~\ref{Thm:Lyapunov-spectrum}}
\label{ss.endThm:Lyapunov-spectrum}
	Let $\mathcal{R}\in \diff^1(M)$ be the residual subset from Theorem~\ref{Thm:construction-sequences}. Consider $f\in\mathcal{R}$,  a nontrivial homoclinic class $H$ of $f$, and 
	$L=(\lambda_1,\dots,\lambda_m)\in \interior (\mathcal L_{\mathrm{per}} (H))$. Fixed any $\delta \in (0,1)$,
	consider  the sequences of saddles $\{p_n\}$ inside $H$, of positive constants $\left\{\gamma_n\right\}$, and of positive integers $\left\{k_n\right\}$ provided by Theorem~\ref{Thm:construction-sequences}.
	Note that $\gamma_{n+1}< \gamma_n/2$. Consider also $\kappa_n \eqdef(1-1/2^n)$ and note that
	$\prod_{n=1}^\infty \kappa_n >0$. Note that in our construction $\orb(p_{n+1})$ is a $(\gamma_n,1-\frac{1}{2^n})$-good approximation of $\orb(p_n)$. 
	For each $p_n$, consider the periodic measure $\mu_n\eqdef \mu_{p_n}$ defined in
	\eqref{e.mup}.
	By Item~\ref{item:good-approximation} of Theorem~\ref{Thm:construction-sequences} and Lemma~\ref{Lem:GIKN-criterion}, the sequence $\{\mu_n\}$ weak$^\ast$-converges to an ergodic measure $\mu$ such that 
	\[
	\supp(\mu)=\bigcap_{n\geq 1}\overline{\bigcup_{j\geq n} \orb(p_j)}.
	\]
	By Item~\ref{item:densness} of Theorem~\ref{Thm:construction-sequences}, one has that $\supp(\mu)=H$.

	It remains to verify that $L(\mu)=L$, that is, $\lambda_i(\mu)=\lambda_i$ for each $i=1,\dots,m$. 
	The proof is similar to~\cite[Claim 4.3]{Cheetal:19}.
	Since $\orb(p_{n+1})$ is a $(\gamma_n,1-\frac{1}{2^n})$-good approximation of $\orb(p_n)$,   there are
	a subset $X_{n}\subset \orb(p_{n+1})$ and a map $\rho_n\colon X_n\rightarrow \orb(p_n)$ as in Definition~\ref{Def:good-approximation}.
	Let 
	\begin{equation}
	\label{e.accumulatedprojection}
	Y_n\eqdef \rho_n^{-1}\circ\cdots\circ\rho_0^{-1}(\orb(p_0))\subset\orb(p_{n+1}) \quad \text{and} \quad 
	Y\eqdef \limsup_{n\rightarrow \infty} Y_n.
	\end{equation}
	By definition,
	\[
	\mu_n (Y_n) \geq 
	 \prod_{\ell=0}^{n} \left(1-\frac{1}{2^\ell}\right).
	\]
	Then, by the compactness of $Y$ and $\mu_n\rightarrow \mu$, one has that 
	\[
	\mu(Y)\geq \limsup_{n\rightarrow \infty}\mu_n(Y_n)\geq \prod_{n=0}^{\infty} \left(1-\frac{1}{2^n}\right)>0.
	\]
	The choice  $\gamma_{n+1}<\frac{\gamma_n}{2}$ implies that $Y \subset B_{2\gamma_n}(\orb(p_n))$ for every $n\in \mathbb{N}$. Thus, by Items~\ref{item:exterior-product} and 
	\ref{item:Lyapunov-spectrum}
	of Theorem~\ref{Thm:construction-sequences}, for every $x\in Y$, every $i=1,\dots,m$, and every $n\in\mathbb{N}$ it holds 
		\[
		\begin{split}
		\left|L_i^{(k_n)}(x,f)-L_{i-1}^{(k_n)}(x,f)-\lambda_i\right|&
		\leq\left|L_i^{(k_n)}(x,f)-L_{i-1}^{(k_n)}(x,f)-\lambda_i(p_n)\right|+\left|\lambda_i(p_n)-\lambda_i\right|\\
		 &<\frac{\delta}{2^n}+\frac{\delta}{2^n}=\frac{\delta}{2^{n-1}}.
	\end{split}
	\]
	Since $\{k_n\}$ converges to $\infty$ and $\mu$ is ergodic, for $\mu$-a.e. $x\in Y$ it holds	
	$$
		\lambda_i(\mu)
		=\lim_{k\rightarrow\infty} L_i^{(k)}(x,f)-L_{i-1}^{(k)}(x,f)\\
		=\lim_{n\rightarrow\infty} L_i^{(k_n)}(x,f)-L_{i-1}^{(k_n)}(x,f)=\lambda_i.
	$$
	Since this  assertion holds for every $i=1,\dots,m$, we get
$L(\mu)=L$, This ends the proof of the theorem.
\qed

\section{Proof of Theorem~\ref{Thm:isolated-ergodic-denssness}}
\label{s.proofThm:isolated-ergodic-denssness}

Let $\mathcal{R}$ be a residual subset of $\diff^1(M)$ in \eqref{e.thesetR}.
	Given any $f\in\mathcal{R}$, consider a nontrivial isolated homoclinic class $H$ of $f$,
	$L=(\lambda_1,\dots,\lambda_m)\in\interior (\mathcal{L}(H))$, and  $\mu \in \mathcal{M}_{\mathrm{inv}} (f)$ 
	with $L(\mu)=L$. 

To prove Theorem~\ref{Thm:isolated-ergodic-denssness}, 
we apply arguments similar to those in the proof of Theorem~\ref{Thm:construction-sequences} to construct a sequence of periodic orbits ${\orb(p_n)}$ satisfying the criterion of Lemma~\ref{Lem:GIKN-criterion}. Then the corresponding sequence of periodic measures ${\mu_{p_n}}$  converges to an ergodic measure $\nu$ with the same Lyapunov spectrum as the given invariant measure $\mu$.
It remains to ensure that, given any neighborhood $\mathcal{U}$ of $\mu$, the periodic measures $\mu_{p_n}$,
and hence the limit measure $\nu$, are all contained in $\mathcal{U}$. We now go to the details. The theorem now follows from the next proposition:

\begin{proposition}
\label{p.l.foreveryneigh}
For every a neighborhood $\mathcal{U}$ of $\mu$ in $\mathcal{M}_{\rm inv}(H)$, there is  
	$\nu_0\in\mathcal{U}\cap \mathcal{M}_{\rm erg}(H)$ such that  $\supp (\nu_0)=H$ and $L(\nu_0)=L$.
	\end{proposition}

\begin{proof}	
Since we consider diffeomorphisms in $\mathcal{R}$, we can apply  Lemma~\ref{lProp:approximation-invariant-measure}. 
Consider the sequence of saddles 
$\{\bar p_n\}$ 
in $H\cap \per(f)$ such that every 
$\mu_{\bar p_n}$ has simple spectrum,
$\mu_{\bar p_n}\rightarrow \mu$, and $L(\mu_{\bar p_n})\rightarrow L$ provided by
Lemma~\ref{lProp:approximation-invariant-measure}.
		
		\medskip	
	\noindent{\bf{The neighborhood $\mathcal{V}(\varphi_1,\dots,\varphi_N;\eta)$ of $\mu$.} }
	By the definition of weak$^\ast$ topology,  there is $\eta>0$ and a finite collection of continuous functions $\left\{\varphi_i\colon H\rightarrow \mathbb{R}\right\}_{1\leq i\leq N}$ such that the open set 
	\[
	\mathcal{V}(\varphi_1,\dots,\varphi_N;\eta)\eqdef \left\{\nu\in\mathcal{M}_{\rm inv}(H)\colon \left|\int\varphi_i d\mu-\int \varphi_i d\nu\right|<\eta, ~\text{for all $1\leq i\leq N$} \right\}
	\]
	is contained in $\mathcal{U}$.

	\medskip	
	\noindent{\bf{The constants $K,\delta$ and the  integer $n_0$.} }
	Take $K>0$ and  an integer $n_0\geq 2$  such that
		\begin{align}\label{Equ:n0}
	\|\varphi_i\|_\infty<K \quad \text{for all $1\leq i\leq N$} \qquad \text{and} \qquad
		2 K \left(1-\prod_{n=n_0}^{\infty} \left(1-\frac{1}{2^n}\right)\right)<\frac{\eta}{8}.
	\end{align}
	As in the proof of Theorem~\ref{Thm:construction-sequences}, take  $\delta>0$ such that 
	 \[
	 B_{4 \delta} (L)=\left\{(\eta_1,\dots,\eta_m)\colon |\eta_i-\lambda_i|<4\delta, \,\, \text{for every}\,\,i=1,\dots,m\right\}\subset \interior ( \mathcal{L}(H)).
	 \]
	 
	\medskip	
	\noindent{\bf{The sequence of periodic orbits $\left\{\orb(p_n)\right\}_{n\geq n_0}$.}}
	Taking $n_0$ large enough and applying  Item~\ref{generic:new-periodic-orbit} of Theorem~\ref{Thm:generic-properties} to the saddle $\bar p_{n_0}$ with simple spectrum, we get a saddle $p_{n_0}\in H$ such that
	\begin{itemize}
		\item the orbit $\orb(p_{n_0})$ has simply spectrum and  is $\frac{\delta}{2^{n_0}}$-dense in $H$;
		
		\item the periodic measure $\mu_{p_{n_0}}\in\mathcal{V}(\varphi_1,\dots,\varphi_N;\frac{\eta}{4})$, or equivalently.
\begin{equation}
\label{e.orequiv}		
\left|\int\varphi_i d\mu_{p_{n_0}}-\int \varphi_i d\mu\right|<\frac{\eta}{4}, \quad \text{for every $1\leq i\leq N$};
\end{equation}
		
    		\item $\lambda_i-\frac{\delta}{2^{n_0}}<\lambda_{i}(p_{n_0})<\lambda_i+\frac{\delta}{2^{n_0}}$ for every $i=1,\dots,m$. Equivalently, $L(p_{n_0})\in B_{{\delta}/{2^{n_0}}}(L)$.
	\end{itemize}

Following the induction process in the proof of Theorem~\ref{Thm:construction-sequences}, 
 we get sequences of periodic orbits $\left\{\orb(p_n)\right\}_{n\geq n_0}$, of constants 
	$\left\{\gamma_n\right\}_{n\geq n_0}\subset (0,1/2)$, and of positive integers $\left\{k_n\right\}_{n\geq n_0}$ 
	with
	$$
	\gamma_{n+1}<\frac{\gamma_n}{2}
	\qquad \mbox{and}
	\qquad
	k_{n+1} > k_n
	$$ 
	satisfying the following
	for every $n\geq n_0$:
	\begin{enumerate}
		\item\label{item:densnessB} 
		the orbit $\orb(p_n)$ has simple spectrum and is $\frac{\delta}{2^n}$-dense in $H$;
		
		\item\label{item:good-approximationB} 
		the orbit is $\orb(p_{n+1})$ is a $(\gamma_n,1-\frac{1}{2^n})$-good approximation of $\orb(p_n)$;
		
		\item\label{item:Lyapunov-spectrumB}  $\lambda_i-\frac{\delta}{2^n}<\lambda_i(p_n)<\lambda_i+\frac{\delta}{2^{n}}$ for each $i=1,\dots,m$, hence $\left\|L(p_n)-L\right\|<\frac{\delta}{2^n}$;
		
		\item\label{item:exterior-productB} 
		for every $n\in\mathbb{N}$ and every $x\in B_{2\gamma_n}(\orb(p_n))$ and every $i=1,\dots,k$, one has
		\[\left|L_i^{(k_n)}(x,f)-L_{i-1}^{(k_n)}(x,f)-\lambda_i(p_n)\right|<\frac{\delta}{2^n}
		\qquad  \text{for  every $i=1,2,\ldots,k$}; 
		\]
		\item\label{item:gamma}
		
		for every $x,y\in H$ with $d(x,y)<2\gamma_{n_0}$, one has 
		\[\left|\varphi_i(x)-\varphi_i(y)\right|<\frac{\eta}{8}, \qquad \text{for every $i=1,\dots,N$}.\]
	\end{enumerate}	
	
	\medskip
	
	\noindent{\bf {The ergodic measure $\nu_0$.} }
	By Lemma~\ref{Lem:GIKN-criterion}, the sequence $\{\mu_{p_n}\}_{n\geq n_0}$ converges to an ergodic 
	measure $\nu_0$ with 
	\[
	\supp(\nu_0)=\bigcap_{n\geq n_0}\overline{\bigcup_{i\geq n} \orb(p_i)}=H,
	\]
	where the second equality is by Item~\ref{item:densnessB}.
	By construction, 
	$L(\nu_0)=L.$ 
	Then the proposition follows from the next claim.
	
	\begin{lemma}\label{lClaim:nu0-close-to-mu}
		$\nu_0\in \mathcal{V}(\varphi_1,\dots,\varphi_N;\eta)\subset \mathcal{U}$.
	\end{lemma}
	\begin{proof}
		As $\nu_0=\lim\limits_{n\rightarrow \infty} \mu_{p_n}$, it is sufficient to see
		that $\mu_{p_n}\in \mathcal{V}(\varphi_1,\dots,\varphi_N;\frac{\eta}{2})$ for all $n> n_0$.
		
		By Item~\ref{item:good-approximationB}, 
		one can see that for each $n> n_0$, the orbit 
		$\orb(p_n)$ is a $\left(\bar\gamma_n,\bar\kappa_n\right)$-good approximation of $\orb(p_{n_0})$, where 
		\[
		\bar\gamma_n\eqdef \sum_{n=n_0}^{n-1}\gamma_n<2\gamma_{n_0}\qquad \text{and}\qquad \bar\kappa_n\eqdef \prod_{n=n_0}^{n-1} \left(1-\frac{1}{2^n}\right).
		\]		
		The estimate in $\bar \gamma_n$ follows from $\gamma_{n+1}<\frac 12 \gamma_n$ and implies 
		$\sum_{n=n_0}^{n-1}\gamma_n<2\gamma_{n_0}$.
			
		To simplify notations, write $\pi_n\eqdef \pi(p_n)$.	
		Arguing now as in the end of the proof of Theorem~\ref{Thm:Lyapunov-spectrum}, we define 
		sets $Z_n$ similarly as the set $Y_n$ in \eqref{e.accumulatedprojection}.
		More precisely,
		consider the
	subsets $X_{n}\subset \orb(p_{n+1})$ and the map $\rho_n\colon X_n\rightarrow \orb(p_n)$ provided by Definition~\ref{Def:good-approximation} and let
	\begin{equation}
	Z_n\eqdef \rho_{n-1}^{-1}\circ\cdots\circ\rho_{n_0}^{-1}(\orb(p_{n_0}))\subset\orb(p_{n}).	
	\end{equation}
In this way we get
		\begin{equation}
		\label{e.inthiswayweget}
		\frac{\# Z_n}{\pi_n}>\bar\kappa_n\eqdef \prod_{n=n_0}^{n-1} (1-\frac{1}{2^n}).
				\end{equation}	
Moreover, 
  \begin{itemize}[leftmargin=0.4cm ]				
			\item  $d(f^j(z),f^j(\rho(z)))<\sum_{n=n_0}^{n-1}\gamma_n<2\gamma_{n_0}$ for every $z\in Z_n$ and every $j=0,1,\dots,\pi_{n_0}$.
			Thus, by Item~\ref{item:gamma}, one has 
		\begin{equation}
		\label{e.comparation}
		\left|\varphi_i(z)-\varphi_i(\rho(z))\right|<\frac{\eta}{8}, \qquad \text{for every $i=1,2,\ldots,N$}.
		\end{equation}
		
			\item $\#\rho^{-1}(z)$ is constant and does not depend on $z\in\orb(p_{n_0})$.
			Thus 
			$$
			\# Z_n=\#\rho^{-1}(z) \pi_{n_0}.
			$$
		\end{itemize}
				
\begin{claim}
\label{cf.estimate}
For every $n \geq n_0$ and $i=1,\dots,N$ it holds
$$ 
\left|\int \varphi_i d \mu_{p_n}- \int \varphi_i d \mu_{p_{n_0}}\right|< \dfrac{\eta}{4}.
$$
\end{claim}

Let us postpone the proof of Claim~\ref{cf.estimate} and complete the proof of the lemma. 
Using the fact and \eqref{e.orequiv}	we get, for each
 $1\leq i\leq N$, 
		\begin{align*}
			\left|\int\varphi_i d\mu_{p_{n}}-\int \varphi_i d\mu\right|\leq \left|\int \varphi_i d \mu_{p_n}- \int \varphi_i d \mu_{p_{n_0}}\right|+\left|\int\varphi_i d\mu_{p_{n_0}}-\int \varphi_i d\mu\right|<\frac{\eta}{2}.
		\end{align*}
		Thus  $\mu_{p_n}\in \mathcal{V}(\varphi_1,\dots,\varphi_N;\frac{\eta}{2})$ for all $n\geq n_0$.
		As $\nu_0$ is the limit of the measures $\mu_{p_n}$ the lemma follows.

\begin{proof}[Proof of Claim~\ref{cf.estimate}]					
Given any $i=1,\dots,N$ write 
\begin{align*}
	\Big|\int \varphi_i d \mu_{p_n}- \int &\varphi_i d \mu_{p_{n_0}}\Big|
	=\left|\frac{1}{\pi_n} \sum_{j=0}^{\pi_n-1} \varphi_i (f^j(p_{n}))
			-\frac{1}{\pi_{n_0}} \sum_{j=0}^{\pi_{n_0}-1} \varphi_i (f^j(p_{n_0}))\right|\\
	=&\left| \frac{1}{\pi_{n}}\left( \sum_{y\in Z_n} \varphi_i(z) + \sum_{x\in \orb(p_n)\setminus Z_n} \varphi_i(x) \right)- 
			\frac{1}{\pi_{n_0}}\sum_{j=0}^{\pi_{n_0}-1} \varphi_i (f^j(p_{n_0}))\right|\\
	\leq & \left| \frac{1}{\pi_{n}}\sum_{y\in Z_n} \varphi_i(z) -\frac{\# Z_n}{\pi_{n}}\frac{1}{\pi_{n_0}}\sum_{j=0}^{\pi_{n_0}-1} \varphi_i (f^j(p_{n_0})) \right|\\
	&+\left| \frac{1}{\pi_{n}}\sum_{x\in \orb(p_n)\setminus Z_n} \varphi_i(x) -  
			\left(1-\frac{\# Z_n}{\pi_{n}}\right)\frac{1}{\pi_{n_0}}\sum_{j=0}^{\pi_{n_0}-1} \varphi_i (f^j(p_{n_0}))\right|\\
			\leq & \frac{1}{\pi_{n}}\left| \sum_{z\in Z_n}\left(\varphi_i(z) -\varphi_i(\rho(z))\right)\right|+ \left(1-\frac{\# Z_n}{\pi_{n}}\right) 2 \left\|\varphi_i\right\|_{\infty}\\
		 \footnotesize{\mbox{(by \eqref{e.comparation}, \eqref{e.inthiswayweget}, \eqref{Equ:n0})}}	<& \frac{\# Z_n}{\pi(p_{n})}\cdot \frac{\eta}{8}+ 2\cdot K\cdot \left(1-\prod_{n=n_0}^{\infty} \left(1-\frac{1}{2^n}\right)\right)\\
	 \footnotesize{\mbox{(by \eqref{Equ:n0})}}	 <
	 & \frac{\eta}{8}+\frac{\eta}{8}=\frac{\eta}{4},
\end{align*}
proving the claim. 
\end{proof}

		Thus, one has for all $1\leq i\leq N$,
		\begin{align*}
			\left|\int\varphi_i d\mu_{p_{n}}-\int \varphi_i d\mu\right|\leq \left|\int \varphi_i d \mu_{p_n}- \int \varphi_i d \mu_{p_{n_0}}\right|+\left|\int\varphi_i d\mu_{p_{n_0}}-\int \varphi_i d\mu\right|<\frac{\eta}{2},
		\end{align*}
		ending he proof of Lemma~\ref{lClaim:nu0-close-to-mu}.			
	\end{proof}
	The proof of Proposition  \ref{p.l.foreveryneigh} is now complete.
	\end{proof}
	
	The proof of Theorem~\ref{Thm:isolated-ergodic-denssness} is now complete. \qed

\section{Interior of Lyapunov spectrum: proof of Theorem~\ref{Thm:periodic-interior-spectrum}}
\label{s.interiorofthespectrum}

To prove Theorem~\ref{Thm:periodic-interior-spectrum}, we use the perturbation
Lemma~\ref{Lem:perturbation-spectrum} below that allows us to increase or decrease prescribed Lyapunov exponents, while preserving homoclinic relations, see  Section~\ref{ss.aperturbationlemma}. The proof of Theorem 
\ref{ss.aperturbationlemma} is completed in Section~\ref{ss.endofThm:periodic-interior-spectrum}.

\subsection{A perturbation lemma to increase/decrease Lyapunov exponents}\label{ss.aperturbationlemma} 

\begin{lemma}\label{Lem:perturbation-spectrum}
	Let $\mathcal{R}$ be the residual  set of $\diff^1(M)$ in \eqref{e.thesetR}. Then 
for every $f\in\mathcal{R}$ and every  saddle $p$ of $f$ with simple spectrum and nontrivial homoclinic 
class the following holds: for every subset $I\subset\{1,\dots,m\}$ and every $C^1$-neighborhood
$\mathcal{U}$ of $f$, there exists $g\in\mathcal{U}$ such that

	\begin{enumerate}
		\item $g$ preserves the periodic orbit $\orb(p,f)$: $g|_{\orb(p,f)}=f|_{\orb(p,f)}$;
		
		\item $g$ has a saddle $q_g^I$ such that $\,\orb(q_g^I,g)$ is homoclinically related to $\orb(p,g)$,
		 and 
		 \item
		 the spectrum of $q^I_g$ for $g$ satisfies
		\[
		\begin{split}
		&\lambda_i(q_g^I,g)>\lambda_i(p,g) \quad \mbox{for every $i\in I$};\\
		&\lambda_i(q_g^I,g)<\lambda_i(p,g) \quad \mbox{for every $i\in \{1,2,\dots,m\}\setminus I$.}
		\end{split}
		\]		
	\end{enumerate} 
\end{lemma}

To prove Lemma~\ref{Lem:perturbation-spectrum} we use the generalized
version of Franks' lemma \cite{Fra:71}  in ~\cite{Gou:07}, which allows 
to perturb the derivative along  hyperbolic periodic orbits while preserving
selected parts of their invariant manifolds.

Given a hyperbolic periodic orbit $\orb(q)$ of $f$ and $\delta>0$, let $ W^{\mathrm{s}}_\delta(\orb(q), f)$ and $W^{\mathrm{u}}_\delta(\orb(q), f)$ be the local stable and unstable manifolds of $\orb(q)$ of size $\delta$.

\begin{lemma}[{\cite[Theorem 2]{Gou:07}}]\label{Lem:Gourmelon-Franks}
	Let 
	$f\in\diff^1(M)$ and
	$q$ be saddle of $f$. Given $\varepsilon>0$
	 consider a one-parameter family of maps
\[
\{A_{j,t}\}_{j=0,\dots,\pi(q)-1,\ t\in[0,1]}
\subset \mathrm{GL}(m,\mathbb{R})
\]
such that 
  \begin{itemize}[leftmargin=0.6cm ]
		\item $A_{j,0}=Df(f^j(q))$ for every $j=0,\dots,\pi(q)-1$;
		
		\item 
		for every $j=0,\dots,\pi(q)-1$ and every $t\in[0,1]$ it holds
$$\max\left\{\|A_{j,t}-Df(f^j(q))\|,\|A_{j,t}^{-1}-Df^{-1}(f^j(q))\|\right\}<\varepsilon;
$$ 		
		\item $A_{\pi(q)-1,t}\circ \cdots \circ A_{0,t}$ is hyperbolic for every $t\in [0,1]$.
	\end{itemize}
	Then for every neighborhood $V$ of $\orb(q)$, every $\delta>0$, and every pair of compact sets $K^{\mathrm{s}}\subset W^{\mathrm{s}}_\delta(\orb(q), f)$ and $K^{\mathrm{u}}\subset W^{\mathrm{u}}_\delta(\orb(q), f)$ disjoint with $V$, there is $g\in\diff^1(M)$ 
	such that
	
	\begin{enumerate}
		\item $g$ is $\varepsilon$-$C^1$-close to $f$ and coincides with $f$ on $\orb(q)$ and outside $V$;
		
		\item $K^{\mathrm{s}}\subset W^{\mathrm{s}}_\delta(\orb(q), g)$ and $K^{\mathrm{u}}\subset W^{\mathrm{u}}_\delta(\orb(q), g)$; 
		
		\item $Dg(g^j(q))=Dg(f^j(q))=A_{j,1}$ for every $j=0,\dots,\pi(q)-1$. 	
		\end{enumerate}
\end{lemma}

\begin{proof}[Proof of Lemma~\ref{Lem:perturbation-spectrum}]
	Write by $\lambda_i=\lambda_i(p)$,  $i=1,\dots,m$. 
	Since $\orb(p)$ has simple spectrum, $T_{\orb(p)}M$ has a $Df$-invariant splitting $T_{\orb(p)}M=E^1\oplus \cdots\oplus E^m$ with 
	one-dimensional bundles.  	 We have 
	\[
	\lambda_i= \frac{1}{\pi(p)}\log\left\|Df^{\pi(p)}|_{E^i_{x}}\right\|
	\qquad \text{for every $x\in\orb(p)$}.
	\]
	As $\lambda_i\neq 0$, because $p$ is hyperbolic,
	there is $\eta_0>0$ such that 
	\[
	\eta_0<\frac{1}{2}\min_{1\leq i\leq m-1}\left\{\lambda_{i+1}-\lambda_i\right\}\qquad 
	\text{and}
	\qquad \eta_0<\min_{1\leq i\leq m}\frac{1}{2}|\lambda_i|.
	\]

For each $j=0, \dots, \pi(p)-1$, we choose
a normal basis $\beta_{f^j(p)}$
 of $ T_{f^j(p)}M$ whose vectors are in the bundles
 $E^i_{f^j(p)}$
with respect to which 
$Df \colon T_{f^j(p)}M \to T_{f^{j+1}(p)}M$
is represented by  a diagonal matrix $A_j$. 
	
	Fix $\gamma\in (0,\gamma_0)$ small enough (to be determined later). By Item~\ref{generic:new-periodic-orbit} of Theorem~\ref{Thm:generic-properties}, 
	there exists a saddle $q$  of $f$ (whose orbit is different from that of $p$) such that
  \begin{itemize}[leftmargin=0.6cm ]
		\item $\orb(q)$ has simple spectrum and is homoclinically related to $\orb(p)$;
		
		\item  $\orb(q)$ is $(\gamma,\frac{1}{2})$-good approximation of $\orb(p)$;
		
		\item $|\lambda_i(q)-\lambda_i(p)|<\frac{\eta_0}{4}$ for every $i=1,\dots,m$.
	\end{itemize}  
	Take points 
	$$
	z^{\mathrm{s}}\in W^{\mathrm{s}}(\orb(q))\pitchfork W^{\mathrm{u}}(\orb(p)) \quad 
	\mbox{and} \quad z^{\mathrm{u}}\in W^{\mathrm{u}}(\orb(q))\pitchfork W^{\mathrm{s}}(\orb(p))
	$$ 
	and  a neighborhood $V$ of $\orb(q)$ such that 
	\begin{equation}
	\label{e.neighV}
	(\orb(p)\cup \{z^{\mathrm{s}},z^{\mathrm{u}}\})\cap \overline{V}=\emptyset.
	\end{equation}
	Since $\orb(q)$ has simple spectrum, there exists
	$Df$-invariant splitting with one-dimensional bundles
	$T_{\orb(q)}M=\widetilde E^1\oplus \cdots\oplus \widetilde E^m$. 
 As in the case of the orbit of $p$, for each $f^j(q)$ we can consider a basis
$\widetilde \beta_{f^j(q)}$ formed by unitary vectors contained in the bundles $\widetilde E^i_{f^j(q)}$. We chose
these basis ``compatible''  with $\beta_{f^j(p)}$ as explained below.

	Let $X \subset \orb(q)$ and $\rho \colon X \to \orb(p)$ be the set and the map given in 
	Definition~\ref{Def:good-approximation} of a $(\gamma,\kappa)$-good approximation.
For $\gamma$ sufficiently small, for each $x \in X$ the splittings
\[
\widetilde E^1_x \oplus \cdots \oplus \widetilde E^m_x
\qquad \text{and} \qquad
E^1_{\rho(x)} \oplus \cdots \oplus E^m_{\rho(x)}
\]
are $\gamma_0$-$C^0$-close.
Thus for points $x\in X$,
this allows us to choose the vectors of the basis $\widetilde \beta_x$ of $T_xM$ 
close to vectors of the basis $\beta_{\rho(x)}$.
Similarly for 
$\widetilde \beta_{f(x)}$ and
$\beta_{f(\rho(x))}$.

	Consider the linear map $Df\colon T_{f^j (q)}M\rightarrow T_{f^{j+1}(q)}M$ 
	and its matrix $B_j$ with respect to the basis $\widetilde \beta_{f^j (q)}$ and $\widetilde \beta_{f^{j+1} (q)}$.
	One has 
	\[
	\lambda_i(q)=\frac{1}{\pi(q)}\log\left\|Df^{\pi(q)}|_{E^i_{q}}\right\|=\frac{1}{\pi(q)}\sum_{j=0}^{\pi(q)-1} \log 
	| B_j|_{E^i_{ f^j (q)}}|,
	\] 
	where $B_j|_{E^i_{ f^j (q)}}$ is the entry $(j,j)$ of the matrix $B_j$.

In what follows, $\exp(\diag\left(t_1, \dots, t_m\right))$ denotes the $m\times m$ diagonal 
	matrix whose diagonal entries are $\text{exp} t_1, \dots, \text{exp}t_m$. 	

\begin{remark}\label{r.compatibility}
Given any $\varepsilon>0$, there is $0<\eta<\eta_0$ satisfying such that for  
every $j=0, \dots, \pi(p)-1$  and every $\zeta\in [0, \eta]$ the following holds

\[
\max\{ \| B_{j} (\zeta) - B_j \|, \| B_{j}(\zeta)^{-1} - B_j^{-1}\|\}
< \varepsilon, \,\, \mbox{where $B_{j}(\zeta)\eqdef
 \exp (\diag(\pm\zeta,\dots,\pm\zeta))\circ B_j$}.
\]
	Moreover, for every $\zeta_{\pi(q)-1}, \dots, \zeta_{0} \in [0, \eta]$ the product
	$B_{\pi(q)-1} (\zeta_{\pi(q)-1}) \cdots B_{0} (\zeta_0)$ is a hyperbolic matrix.
	\end{remark}

We now define the a one-parameter family of matrices
$\{B_{j,t}\}_{0\le j\le \pi(q)-1,\ t\in[0,1]}$.
Let $\mathbb{1}_I(\cdot)$ denote the characteristic function of $I$
(that is, $\mathbb{1}_I(i)=1$ if $i\in I$ and $\mathbb{1}_I(i)=0$ if $i\notin I$). Then
  \begin{itemize}[leftmargin=0.6cm ]
		\item if $f^j(q)\in X$, then 
		\begin{equation}\label{e.jinI}
		B_{j,t}\eqdef \exp (\diag\left((2\cdot\mathbb{1}_I(1)-1)  t\eta,\dots, (2\cdot\mathbb{1}_I(m)-1) t\eta)\right)\circ B_j,
		\end{equation}
		\item if $f^j(q)\in \orb(q)\setminus X$, then $B_{j,t}\eqdef B_j$.
	\end{itemize}

	By the choice of $\eta$ done in Remark~\ref{r.compatibility} and the definition of $B_{j,t}$,
	we have that $\|B_{j,t}-B_j\|<\varepsilon$ and $\|B_{j,t}^{-1}-B_j^{-1}\|<\varepsilon$ and
	the compositions  $B_{\pi(q)-1,t}\circ\cdots \circ B_{0,t}$
	are hyperbolic for every $t\in [0,1]$. Thus, we 
can apply Lemma~\ref{Lem:Gourmelon-Franks}, to obtain $g$ which is $\varepsilon$-$C^1$-close to $f$ such that $g$ coincides with $f$ on $\orb(q)$ and outside its
	neighborhood 
	$V$ satisfying \eqref{e.neighV},
	and $Dg(g^j(q))=B_{j,1}$ for every $j=0,\dots,\pi(q)-1$. Thus $\orb(q,g)$ and $\orb(p,g)$ are still homoclinically related.  Letting $q_g^I=q$, this ends the proof of Items 1 and 2.
	
	It remains to verify that every $\lambda_i(q,g)$ satisfies Item 3 in the lemma.
	Note that every $B_{j,1}$ is diagonal. There are two cases,  when $i\in I$, from 
	the definition of $B_{j,1}$ in
	\eqref{e.jinI}, $|\lambda_i(q)-\lambda_i(p)|<\frac{\eta_0}{4}$, and $\frac{\#X}{\pi(q)}\ge \frac{1}{2}$ 
	it follows
	\[
	\lambda_i(q,g)=\frac{1}{\pi(q)}\sum_{j=0}^{\pi(q)-1} \log \|B_{j,1}|_{E^i}\|\geq \lambda_{i}(q,f)+\frac{\eta_0}{2}
	>\lambda_{i}(p,f).
	\] 
	When  $i\in \{1,\dots,m\}\setminus I$, again from  \eqref{e.jinI}, 
	$|\lambda_i(q)-\lambda_i(p)|<\frac{\eta_0}{4}$, and $\frac{\#X}{\pi(q)}\ge \frac{1}{2}$ it follows
	\[
	\lambda_i(q,g)=\frac{1}{\pi(q)}\sum_{j=0}^{\pi(q)-1} \log \|B_{j,1}|_{E^i}\|<\lambda_{i}(q,f)-\frac{\eta_0}{2}<\lambda_{i}(p,f).
	\]
	The proof of the lemma is now complete.
\end{proof}

\subsection{End of the proof of Theorem~\ref{Thm:periodic-interior-spectrum}}
\label{ss.endofThm:periodic-interior-spectrum}
	 To define the residual set in the theorem,  we need some preliminaries.	 
		Take a countable basis $\mathfrak{O}=\left\{O_n\right\}_{n\in\mathbb{N}}$ of $M$ and let 
	$\mathfrak{U}=\left\{U_n\right\}_{n\in\mathbb{N}}$ be the (countable) collection of the finite unions of sets 
	in $\mathfrak{O}$.
	Let $\mathcal{P}(m)$ the parts of the set $\{1,\dots,m\}$.
	 For each $n\in\mathbb{N}$ and $I\in \mathcal{P}(m)$, define the following subsets $\mathcal{H}_{n,I}$ and $\mathcal{N}_{n,I}$ of $\diff^1(M)$:
  \begin{itemize}[leftmargin=0.6cm ]
		\item $\mathcal{H}_{n,I}$ consists of the diffeomorphis $h$ having a pair of saddles
				$q_1,q_2$ whose orbits are pairwise disjoint and such that
				\begin{itemize}
			\item[--]   their orbits  have simple spectrum and are homoclinically related;
			
			\item[--] 	$\orb(q_1,h)\subset U_n$;
			
			\item[--] $\lambda_i(q_1,h)>\lambda_i(q_2,h)$ if $i\in I$, and $\lambda_i(q_1,h)<\lambda_i(q_2,h)$ if 
			$i\not \in I$.
		\end{itemize}
		
		\item $\mathcal{N}_{n,I}=\diff^1(M)\setminus \overline{\mathcal{H}_{n,I}}$.
	\end{itemize}
	Since the properties in the definition of $\mathcal{H}_{n,I}$ are all $C^1$-open, thus $\mathcal{H}_{n,I}$ is an open subset of $\diff^1(M)$. Thus, 
	$\mathcal{H}_{n,I}\cup \mathcal{N}_{n,I}$ is an open and dense subset of $\diff^1(M)$.
	Let $\mathcal{R}$ be the residual subset of $\diff^1(M)$  in \eqref{e.thesetR}. 
	By construction, 
	\begin{equation}
	\label{e.theresidualset}
	\mathcal{R}'\eqdef \mathcal{R}\cap \bigcap_{I\in \mathcal{P}(m)}\,\,\, \bigcap_{n\in\mathbb{N}} (\mathcal{H}_{n,I}\cup \mathcal{N}_{n,I})
	\end{equation}
	 is a residual subset  of $\diff^1(M)$. We will prove that $\mathcal{R}'$ satisfies the conclusions in Theorem~\ref{Thm:periodic-interior-spectrum}.
	
	Consider $f\in \mathcal{R}'$, a nontrivial homoclinic class $H$ of $f$, and 
	a saddle
	$p\in H$ with simple spectrum 
$L(p)=\left(\lambda_1,\dots,\lambda_m\right)$.
Let $U_n\in\mathfrak{U}$ be an isolating neighborhood of $p$, 
recall \eqref{e.isolatingblock}.

\begin{remark}\label{r.thesetU0}
By the hyperbolicity of $\orb (p)$, there is a $C^1$-neighborhood 
of $f$ in $\diff^1(M)$ such that, for every $h\in\mathcal{U}_0$, the continuation $p_h$ of $p$ is
well defined, has simple spectrum, and $U_n$ is an isolating
neighborhood of $\orb(p_h,h)$.
\end{remark}

	\begin{claim*}
		 For every subset $I\in \mathcal{P}(m)$, it holds $f\in\mathcal{H}_{n,I}$.
    \end{claim*}
    
\begin{proof}
Let $\mathcal{U}_0$ as in Remark~\ref{r.thesetU0}.
    	Take any neighborhood $\mathcal{U}$ of $f$. By Lemma~\ref{Lem:perturbation-spectrum}, there exists $g\in\mathcal{U}\cap\mathcal{U}_0$ that preserves the hyperbolic periodic orbit  $\orb(p,f)$ and has a hyperbolic periodic point $q_g^I$ with simple spectrum such that 
	\begin{equation}
	\label{e.greaterandsmaller}
	\lambda_i(q_g^I,g)>\lambda_i(p,g)
	\quad \text{ if  }i\in I, \qquad
	\lambda_i(q_g^I,g)<\lambda_i(p,g) \quad \text{ if } i\not\in I.
	\end{equation}
Thus	 there is a neighborhood $\mathcal{V}\subset\mathcal{U}_0$ of $g$ such that every $h\in\mathcal{V}$ 
has a hyperbolic periodic point $q_h^I$ homoclinically related to $p_h$ 
satisfying \eqref{e.greaterandsmaller}.
This implies that $g\in \mathcal{H}_{n,I}$. As $g \in \mathcal{U}$ and this neighborhood of $f$ can be taken arbitrarily small, it follows that  $f\in\overline{\mathcal{H}_{n,I}}$. 

Finally, note that $f\in \mathcal{R}'$ implies that $f\in \mathcal{H}_{n,I}\cup \mathcal{N}_{n,I}$. As $\mathcal{N}_{n,I}=\diff^1(M)\setminus \overline{\mathcal{H}_{n,I}}$, one concludes that $f\in\mathcal{H}_{n,I}$, proving the claim.
    \end{proof}
	
Consider any $I\in \mathcal{P}(m)$. Since $f\in\mathcal{H}_{n,I}$, there is a hyperbolic periodic point $q^I_f$ of $f$ with simple spectrum which is homoclinically related to $p$ satisfying \eqref{e.greaterandsmaller}. As a consequence,  we get a family of 
periodic orbits $\{\orb(q^I_f)\}_{I\in \mathcal{P}(m))}$ contained in $H$ such that the convex hull of $\{L(q^I_f)\colon I\in \mathcal{P}(m)\}$ contains $L(p)$ in its interior. By Item~\ref{generic:compact-convex} of Theorem~\ref{Thm:generic-properties}, the convex hull of $\{L(q^I_f)\colon I\in \mathcal{P}(m)\}$  is contained in $\mathcal{L}(H)$. Thus $L(p)\in\interior (\mathcal{L}(H))$, 	 ending the proof of  the theorem. \qed

\section{Lyapunov graph: Proof of Theorem~\ref{Thm:application-Bochi-Bonatti}}

	The proof of  Theorem~\ref{Thm:application-Bochi-Bonatti} follows by combining~\cite[Theorem 2]{BocBon:12} and~\cite[Theorem 3.5(a)]{AbdBonCro:11} with a standard Baire category argument, as the one in the proof of Theorem~\ref{Thm:periodic-interior-spectrum}. We provide only a sketch.	
	
	Let $\mathcal{R}$ be the residual subset of $\diff^1(M)$ in \eqref{e.thesetR}.
	Item~\ref{generic:lyapunovvector} in  Theorem~\ref{Thm:generic-properties}  implies that
 for every $f\in \mathcal{R}$, every nontrivial isolated homoclinic class 
	$H$, and every $\mu \in \mathcal{M}_{\mathrm{inv}} (H)$
there is a sequence of periodic orbits $\{\orb(q_n)\}$ whose periodic measures 
$\mu_{q_n}$ converges to $\mu$,  $\orb(q_n)$ converges to $\supp(\mu)$, and $L(q_n)$ converges to $L(\mu)$. 
Thus $L(\mu)\in\mathcal{L}(H)$.  

Bearing in mind that by Item  \ref{generic:homorelated}
in Theorem \ref{Thm:generic-properties} the class $H$ has a continuation for nearby 
diffeomorphisms $g$ we will denote them by $H_g$.

By~\cite[Theorem 2]{BocBon:12}
%\footnote{In~\cite{BocBon:12}, the authors used the notion of \emph{Lyapunov graph}.}, 
	for any $\varepsilon>0$, there are  diffeomorphisms $g$ $C^1$-$\varepsilon$-close to $f$
	which  preserve $\orb(q_n)$ for large $n$ and such that $L(\mu_{q_n},g)$
	is simple and is $\varepsilon$-close to 
	 the average of Lyapunov spectrum $\widehat L(\mu)$ of $\mu$, see
 \eqref{e.Lhatmu}.
	
	Using now a  Baire argument similar to the one in  
	the proof of Theorem~\ref{Thm:periodic-interior-spectrum}, 
	we get a residual set $\mathcal{R}''$ (which is the intersection of $\mathcal{R}$ with another residual set) 
	consisting of
	diffeomorphisms $g$
	having hyperbolic periodic orbits $\orb(p_g,g)$ in $H_g$ (since 
	$H_f$ is isolated and hence $H_g$ is also isolated) such that $L(\mu_{p_g},g)$ is simple and is $\varepsilon$-close to $\widehat L(\mu)$. By Item~\ref{generic:compact-convex} of Theorem~\ref{Thm:generic-properties}, one has that
	$$
	\left\{L\colon L=t\cdot L(\mu)+(1-t)\cdot L(\mu_{p_g},g), t\in[0,1]\right\} \subset
	\mathcal{L}(H_g).
	$$ 
	Letting $\varepsilon\to 0$ and since the set  $\mathcal{L}(H)$
	is defined as a closure, it follows
	\[
	\left\{L\colon L=t\cdot L(\mu)+(1-t)\cdot \widehat L, t\in[0,1]\right\}\subset\mathcal{L}(H_g),
	\]
	proving the theorem.

\newcommand{\etalchar}[1]{$^{#1}$}
\bibliographystyle{alpha}

\end{document}